%% file: parallelSDC.tex
\RequirePackage{amsmath}
\documentclass[twocolumn]{svjour3}
\usepackage{graphicx}
\usepackage{amssymb}
\usepackage{color}

\usepackage{subcaption}
\newcommand{\matr}[1]{\mathbf{#1}}
\newcommand{\Qmat}{\matr{Q}}
\newcommand{\QDmat}{\matr{Q}_\Delta}

\newcommand{\vect}[1]{\boldsymbol{#1}}
\newcommand{\dt}{\Delta t}

\usepackage{pgf}

\begin{document}
\title{Parallelizing spectral deferred corrections across the method}
\author{Robert Speck}
\institute{R. Speck, r.speck@fz-juelich.de \at J\"ulich Supercomputing Centre, Forschungszentrum J\"ulich GmbH}
\date{Received: date / Revised version: date}
\maketitle
\begin{abstract}

% 
% While large-scale parallel-in-time integration methods that allow to integrate many time-steps simultaneously can provide a significant speedup beyond the limits of standard parallelization techniques, their implementation and tuning are often far from straightforward.

% For initial value problems, approaches to parallelize the integration of each time-step separately only allow for small-scale parallelization. However, these strategies to ``parallelize across the method'' often show promising parallel efficiencies and are easy to apply.
% 
% While parallelization across the method for initial values problems, i.e.\ approaches to parallelize the integration of each time-step individually, only allow for small-scale parallelization, their parallel efficiency as well as their ease of applicability are very favorable and promising properties. 

In this paper we present two strategies to enable ``parallelization across the method'' for spectral deferred corrections (SDC).
Using standard low-order time-stepping methods in an iterative fashion, SDC can be seen as preconditioned Picard iteration for the collocation problem. 
Typically, a serial Gau\ss-Seidel-like preconditioner is used, computing updates for each collocation node one by one.
The goal of this paper is to show how this process can be parallelized, so that all collocation nodes are updated simultaneously.
The first strategy aims at finding parallel preconditioners for the Picard iteration and we test three choices using four different test problems.
For the second strategy we diagonalize the quadrature matrix of the collocation problem directly.
In order to integrate non-linear problems we employ simplified and inexact Newton methods.
Here, we estimate the speed of convergence depending on the time-step size and verify our results using a non-linear diffusion problem.

%end of abstract
\end{abstract}

\section{Introduction}

Implicit integration methods based on collocation are an attractive approach to solve initial value problems numerically. 
Depending on the choice of the collocation or quadrature nodes, they feature near-ideal or even ideal (for Gau\ss-Legendre nodes) convergence orders and typically have very advantageous stability properties. 
However, solving the dense and fully coupled collocation problem directly is prohibitively expensive in most cases:
For $M$ collocation nodes and an $N$-dimensional system of ordinary differential equations (ODEs), a system of size $MN\times MN$ has to be solved.
Thus, an iterative strategy is favorable, where instead of the full system only $M$ smaller systems of size $N\times N$ need to be solved for each iteration.

\bigskip

Such an approach is given by the so-called ``spectral deferred correction methods'' (SDC), introduced in~\cite{DuttEtAl2000}.
After casting the initial value problem into its Picard form, a provisional solution to the integral problem is computed using a standard time-stepping method, typically the explicit or the implicit Euler scheme. 
Then, this provisional solution is corrected using a sequence of error integral equations, which are also solved using one of the standard methods. 
This way, a higher-order time-stepping method can be obtained simply by using low-order methods repeatedly.
Xia et al.\ showed in~\cite{ShuEtAl2007} that each iteration or ``sweep'' of SDC can raise the order by one up to the order of the underlying collocation formula.
The iterative structure of SDC has been proven to provide many opportunities for algorithmic and mathematical improvements: convergence can be accelerated by GMRES~\cite{HuangEtAl2006}, IMEX splitting with high orders of accuracy is possible~\cite{Minion2003,RuprechtSpeck2016} and work can be shifted to coarser, less expensive levels to improve the efficiency of SDC~\cite{speck2015multi}.
In the last decade, SDC has been applied e.g.\ to gas dynamics and incompressible or reactive flows~\cite{BouzarthMinion2010,LaytonMinion2004,Minion2004} as well as to fast-wave slow-wave problems~\cite{RuprechtSpeck2016} or particle dynamics~\cite{winkel2015high}.

\bigskip

One of the key features of such an iterative approach for time-stepping, though, is that these approaches can be used to enable efficient parallel-in-time integration.
Using SDC, the ``parallel full approximations scheme in space and time'' (PFASST) by Emmett and Minion~\cite{EmmettMinion2012} allows to integrate multiple time-steps simultaneously by using SDC sweeps on a space-time hierarchy.
This ``parallelization across the steps'' approach~\cite{Burrage1997} targets large-scale parallelization on top of saturated spatial parallelization of partial differential equations (PDEs), where parallelization in the temporal domain acts as a multiplier for standard parallelization techniques in space.
In contrast, ``parallelization across the method'' approaches~\cite{Burrage1997} try to parallelize the integration of each time-step individually.
While this typically results in small-scale parallelization in the time-domain, parallel efficiency and applicability of these methods are often more favorable.
Most notable, the ``revisionist integral deferred correction method'' (RIDC) by Christlieb et al.~\cite{ChristliebEtAl2010} makes use of integral deferred corrections (which are indeed closely related to SDC) in order to compute multiple iterations in a pipelined way.
Also, in~\cite{JacksonEtAl1995,VanderHouwen1990,VanderHouwen1993} parallel Runge-Kutta methods were investigated and we refer to~\cite{Burrage1993} for more examples.

\bigskip

In this paper, we present two approaches to parallelize SDC across the method, allowing to compute the update for all collocation nodes simultaneously.
First, we make use of parallel preconditioners by rewriting SDC as preconditioned Picard iteration, following the ideas of~\cite{HuangEtAl2006,Weiser2014,RuprechtSpeck2016}.
We explore ideas appearing in the context of Runge-Kutta methods~\cite{VanderHouwen1990,VanderHouwen1993} and investigate three different choices, all of which enable parallelization across the nodes.
Second, we diagonalize the quadrature matrix of the collocation problem and show how this idea can be extended to non-linear problems using simplified and inexact Newton methods~\cite{Jay2000}.
We estimate the speed of convergence depending on the time-step size and verify our results using a non-linear diffusion problem.
This second approach is closely related to the diagonalization-based parallelization strategy presented independently in~\cite{Gander2016}, which uses this technique to enable larger-scale parallelization across the steps.
We note that while methods like PFASST target distributed-memory parallelization, both approaches presented here are best implemented using shared-memory parallelization.

\section{Spectral deferred corrections}
For ease of notation we consider a scalar initial value problem
\begin{align*}
  u_t = f(u),\quad u(0) = u_0
\end{align*}
with $u(t), u_0, f(u) \in\mathbb{R}$.
For an interval $[t_0,t_1]$, we rewrite this in Picard formulation as
\begin{align*}
  u(t) = u_0 + \int_{t_0}^t f(u(s))ds,\ t\in[t_0,t_1],
\end{align*}
Introducing $M$ quadrature nodes $\tau_1,...,\tau_M$ with $t_l \le \tau_1 < ... < \tau_M = t_{l+1}$, we can approximate the integrals from $t_l$ to these nodes $\tau_m$ using spectral quadrature like Gau\ss-Radau or Gau\ss-Lobatto quadrature, such that
\begin{align*}
  u_m - \sum_{j=1}^Mq_{m,j}f(u_j) = u_0 
\end{align*} 
where $u_m \approx u(\tau_m)$, $\dt = t_{1}-t_0$ and $q_{m,j}$ represent the quadrature weights for the interval $[t_0,\tau_m]$ such that
\begin{align*}
  \sum_{j=1}^Mq_{m,j}f(u_j)\approx\int_{t_0}^{\tau_m}f(u(s))ds.
\end{align*}
We can now combine these $M$ equations into one system of linear or non-linear equations with
\begin{align}\label{eq:coll_prob}
  \left(\matr{I} - \dt\Qmat\vect{F}\right)(\vect{u}) = \vect{u}_0
\end{align}
where $\vect{u} = (u_1, ..., u_M)^T \approx (u(\tau_1), ..., u(\tau_M))^T\in\mathbb{R}^M$, $\vect{u}_0 = (u_0, ..., u_0)^T\in\mathbb{R}^M$, $\Qmat = (q_{ij})_{i,j}\in\mathbb{R}^{M\times M}$ is the matrix gathering the quadrature weights and the vector function $\vect{F}$ is given by $\vect{F}(\vect{u}) = (f(u_1), ..., f(u_M))^T\in\mathbb{R}^M$.
This system of equations is called the ``collocation problem'' and it is equivalent to a fully implicit Runge-Kutta method, where the matrix $\matr{Q}$ contains the entries of the corresponding Butcher tableau.
We note that for $f(u) \in\mathbb{R}^N$, we need to replace $\Qmat$ by $\Qmat\otimes\matr{I}_N$.

\bigskip

This system of equations is dense and a direct solution is not advisable, in particular if the right-hand side of the ODE is non-linear.
Using SDC, this problem can be solved iteratively and we follow~\cite{HuangEtAl2006,Weiser2014,RuprechtSpeck2016} to present SDC as preconditioned Picard iteration for the collocation problem~\eqref{eq:coll_prob}.
Standard, unmodified Picard iteration is given by
\begin{align*}
  \vect{u}^{k+1} = \vect{u}^{k} + \left(\vect{u}_0 - \left(\matr{I}_{MN} - \dt\matr{Q}\vect{F}\right)\right)\left(\vect{u}^k\right)
\end{align*} 
for $k = 0, ... K$.
For very small $\dt$, this indeed converges to the solution of~\eqref{eq:coll_prob}.
In order to increase range and speed of convergence, we now precondition this iteration.
The standard approach to preconditioning is to define an operator $\matr{P}$ which is easy to invert but also close to the operator of the system.
For SDC, we now choose a simpler quadrature rule for the preconditioner. 
In particular, the resulting matrix $\QDmat$ gathering the weights of this rule is a lower triangular matrix, such that solving the system can be easily done by forward substitution. 
We write
\begin{align}\label{eq:sdc_iteration}
  \left(\matr{I} - \dt\QDmat\vect{F}\right)(\vect{u}^{k+1}) = \vect{u}_0 + \dt(\Qmat-\QDmat)\vect{F}(\vect{u}^{k})
\end{align}
and the operator $\matr{I} - \dt\QDmat\vect{F}$ is then called the SDC preconditioner.
The matrix $\QDmat$ is typically given by the implicit Euler method which corresponds to the right-hand side rule in terms of integration with
\begin{align*}
  \QDmat^{\mathrm{IE}} = \frac{1}{\Delta t}\begin{pmatrix} \Delta \tau_1 & \\ \Delta \tau_1 & \Delta \tau_2 & \\ \vdots & \vdots & \\ \Delta \tau_1 & \Delta \tau_2 & \ldots & \Delta \tau_M \end{pmatrix},
\end{align*}
where $\Delta\tau_m = \tau_{m} - \tau_{m-1}$ for $m=2,...,M$ and $\Delta\tau_1 = \tau_{1} - t_0$ or, using the LU decomposition of $\matr{Q}^T$~\cite{Weiser2014}, by
\begin{align*}
  \QDmat^{\mathrm{LU}} = \matr{U^T}\quad \text{for}\quad \Qmat^T = \matr{L}\matr{U}.
\end{align*}
This choice, sometimes also called the LU trick, is very well suited for stiff problems and has become the de-facto standard choice for $\QDmat$.

\bigskip

Yet, common to these and most other choices is the fact that $\QDmat$ is a lower triangular matrix, so that solving~\eqref{eq:sdc_iteration} can be done only in a serial, Gau\ss-Seidel-like way: first solve for $u^{k+1}_1$ using the initial value $u_0$, then for $u^{k+1}_2$ using $u_1^{k+1}$ and so on.  
In order to introduce parallelism across the quadrature nodes, we investigate two strategies in the following: (A) choose a parallel preconditioner and (B) diagonalize the quadrature matrix.

\section{Parallel preconditioning}

The first idea to parallelize SDC over the quadrature nodes is quite obvious: instead of following a Gau\ss-Seidel-like approach, we try to find suitable matrices $\QDmat$ which only have entries on the the diagonal, i.e.\ which allow to follow a Jacobian-like approach.
To this end, we identify three candidates: 
\begin{enumerate}
  \item take the diagonal of $\Qmat$, i.e. 
  \begin{align}\label{eq:QDmat_Qpar}
    \QDmat^{\mathrm{Qpar}} = \mathrm{diag}(q_{11}, ..., q_{mm}, ..., q_{MM}),
  \end{align}
  \item use Euler steps from $t_0$ to $\tau_m$, i.e. 
  \begin{align}\label{eq:QDmat_IEpar}
    \QDmat^{\mathrm{IEpar}} = \mathrm{diag}(\tau_1-t_0, ..., \tau_m-t_0, ..., \tau_M-t_0),
  \end{align}
  \item minimize the spectral radius of $\matr{I} - \QDmat^{-1}\Qmat$, i.e.
  \begin{align}\label{eq:QDmat_MIN}
    \QDmat^{\mathrm{MIN}} = \mathrm{diag}(\vect{\hat{q}})
  \end{align}
  with
  \begin{align*}
    \vect{\hat{q}} = \mathrm{argmin}_{\vect{q}\in\mathbb{R}^M}\ \rho(\matr{I} - \mathrm{diag}(\vect{q})\Qmat)
  \end{align*}
\end{enumerate}

While the first two approaches are obvious candidates and straightforward to compute, the third one is more involved.
For the linear test problem $u_t = \lambda u$, SDC has an iteration matrix $\matr{K}$ with
\begin{align*}
  \matr{K} = \lambda\dt\QDmat\left(\matr{I} - \lambda\dt\QDmat\right)^{-1}\left(\QDmat^{-1}\Qmat - \matr{I}\right),
\end{align*}
see e.g.~\cite{RuprechtSpeck2016}.
While the first factors all depend on the ``space''-problem parameter $\lambda$ the last factor does not and can therefore be modified independently of the spatial problem at hand.
It also corresponds (up to the sign) to the stiff limit of the iteration matrix, i.e.~for $\lvert\lambda\dt\rvert\rightarrow\infty$ we have $\matr{K}\rightarrow\matr{I} - \QDmat^{-1}\Qmat$, see~\cite{RuprechtSpeck2016}.
We choose to minimize the spectral radius, because the hope is that in this case strong damping of the stiff iteration error components is achieved, see~\cite{VanderHouwen1993} for more details on this matter.

\bigskip

However, to the best of our knowledge there exists no analytic expression for the eigenvalues of $\matr{Q}$ or $\mathrm{diag}(\vect{q})\Qmat$, so that the computation of this minimizer has to be done numerically. 
We use the Nelder-Mead algorithm as implemented by SciPy v.0.18.1~\cite{scipy} in the ``optimize'' package. 
 
\begin{figure}[t]
  \centering
  \begin{subfigure}[b]{\columnwidth}
    \centering
    \includegraphics[width=0.95\columnwidth]{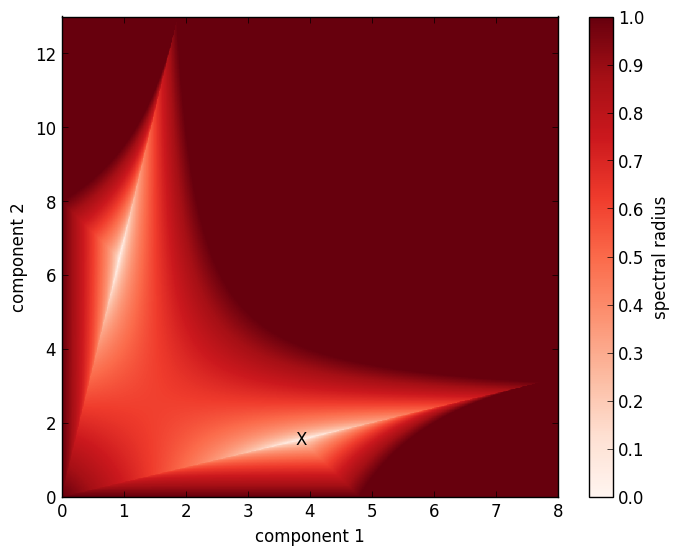}
    \caption{Full region}
    \label{fig:parallelSDC_minimizer_full}
  \end{subfigure}
  \par\medskip
  \begin{subfigure}[b]{\columnwidth}
    \centering
    \includegraphics[width=0.95\columnwidth]{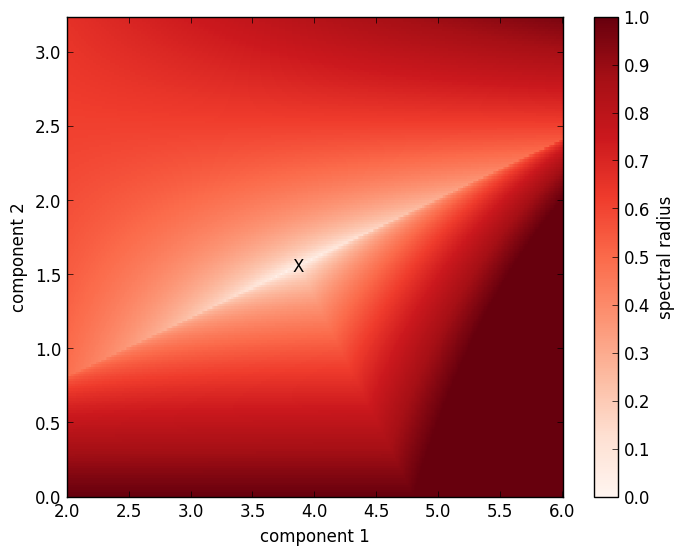}
    \caption{Zoom in}
    \label{fig:parallelSDC_minimizer_zoom}
  \end{subfigure}
  \caption{Spectral radius of $\matr{I}_2-\QDmat^{-1}\Qmat$ for a diagonal matrix $\QDmat\in\mathbb{R}^{2\times 2}$ consisting of different components $1$ (x-axis) and $2$ (y-axis). The X indicates the result of the Nelder-Mead optimization. Note that both axes are scaled differently. Upper: Full domain, lower: zoom into a region of interest.}
  \label{fig:minimzer} 
\end{figure}

To get a first impression of this approach, we consider $M=2$ Gau\ss-Radau nodes and investigate the values of $\rho(\matr{I} - \mathrm{diag}(\vect{q})\Qmat)$ for different $\vect{q} = (q_1,q_2)^T$.
In Figure~\ref{fig:minimzer} we let component $1$ (i.e.~$q_1$) vary between $0$ and $8$ and component $2$ (i.e.~$q_2$) between $0$ and $13$. 
Further experiments not shown here revealed that outside of this region the spectral radius is greater than $1$.
Figure~\ref{fig:parallelSDC_minimizer_full} shows that there are two regions of interest shown in light colors, where a minimum can be expected. 
For Figure~\ref{fig:parallelSDC_minimizer_zoom}, we zoom into the lower region, where the local minimum value computed by the Nelder-Mead optimization with starting value $\vect{q}_0 = (1,1)^T$ is located.
Note that for e.g.~$\vect{q}_0 = (1,2)^T$ the upper local minimum is found, which is about $2.5$ times smaller than the one from the lower region ($2.6\cdot10^{-5}$ vs.~$6.5\cdot10^{-5}$).
Already for $2$ nodes we can see that the regions of small spectral radii as well as the location of the minima are far from trivial.

\bigskip

Now, what is the best choice for $\QDmat$?
``Best'' in our context means that for this choice the corresponding SDC iteration converges about as fast as the standard choices $\QDmat^{\mathrm{IE}}$ and $\QDmat^{\mathrm{LU}}$ for the problem at hand.
Clearly, this is highly problem dependent, but even worse, the same argument which prevented us from finding the minimizer of the spectral radius analytically apply also to finding the best diagonal matrix $\QDmat$, since there is no closed form of the eigenvalues or the norm of any matrix related to $\Qmat$.
Thus, we choose four test problems, two linear and two non-linear, to quantify the impact of the three different $\QDmat$:

\begin{description}
  
  \item[\textbf{Problem A}] Heat equation with $\nu>0$:
  \begin{align*}
    u_t &= \nu\Delta u\quad\text{on } [0,1]\times[0,T],\\
    u(0,t) &= 0,\ u(1,t) = 0,\\
    u(x,0) &= \sin(2\pi x)
  \end{align*}
  \item[\textbf{Problem B}]\label{desc} Advection equation with $c>0$:
  \begin{align*}
    u_t &= c\nabla u\quad\text{on } [0,1]\times[0,T],\\
    u(0,t) &= u(1,t),\\
    u(x,0) &= \sin(2\pi x)
  \end{align*}
  \item[\textbf{Problem C}] Van der Pol oscillator with $\mu>0$:
  \begin{align*}
    u_t &= v,\ v_t = \mu(1-u^2)v - u\quad\text{on } [0,T],\\
    u(0) &= 2,\ v(0) = 0
  \end{align*}
  \item[\textbf{Problem D}] Nonlinear diffusion of Kolmogorov-Pet\-rov\-skii-Piskunov type~\cite{Feng2008481} with $\lambda_0>0$:
  \begin{align}
    \begin{split}
      u_t &= \Delta u + \lambda^2_0 u(1-u^\nu)\quad \text{on } \mathbb{R}\times[0,T],\\
      u(x,0) &= \left(1 + (2^{\nu/2}-1)e^{-(\nu/2)\delta x}\right)^{-\frac{2}{\nu}}
    \end{split}
  \end{align}
  and constants $\delta>0$ and $\nu\in\mathbb{N}$ ($\nu=1$ is used here, $\delta$ can be found in~\cite{Feng2008481}).
\end{description}

For all runs we choose $M=3$ Gau\ss-Radau nodes, $T=\dt=0.1$ and a residual tolerance of $10^{-8}$. 
For Problems A and D we choose $N=63$ degrees of freedom and for B $N=64$.

\bigskip

In Figures~\ref{fig:prec_heat}-\ref{fig:prec_fisher} we show the number of iterations for five different choices of $\QDmat$: the implicit Euler and the LU trick as references as well as the three diagonal matrices defined in~\eqref{eq:QDmat_Qpar},\eqref{eq:QDmat_IEpar} and \eqref{eq:QDmat_MIN}.
For each problem, we vary the parameter given in the description above to change the characteristics or stiffness of the problem.
The first thing to notice is that in almost all cases $\QDmat^{LU}$ is the best choice, in particular if for a problem the convergence of SDC is required to be roughly the same across all parameter values. 
For non-stiff problem, i.e. for small values of the parameters, however, the diagonal matrices work equally well.
Especially $\QDmat^{\mathrm{MIN}}$ is capable of yielding convergence as fast or sometimes even faster than $\QDmat^{\mathrm{LU}}$, if the parameter is small enough.
The other choices, namely $\QDmat^{\mathrm{IEpar}}$ and $\QDmat^{\mathrm{Qpar}}$ perform reasonably well for small parameters, too, but they lead to drastically increased numbers of iterations for larger parameters.
In this regime, the LU trick shows its strength and, not surprisinlgy, this is precisely the way it was designed~\cite{Weiser2014}.
Only for the two non-linear problems and most notably for Problem D we see that $\QDmat^{\mathrm{MIN}}$ is a favorable choice.
We finally note that results look very similar and in parts even slighty better for $M=5$ nodes.

\begin{figure}[p]
  \centering
  \begin{subfigure}[b]{0.9\columnwidth}
    \centering
    \input{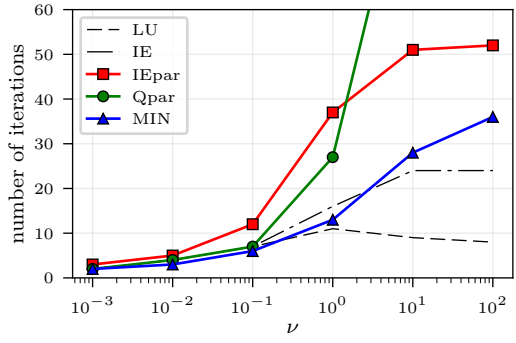}
    \caption{Heat equation}
    \label{fig:prec_heat}
  \end{subfigure}
  \par\medskip
  \begin{subfigure}[b]{\columnwidth}
    \centering
    \input{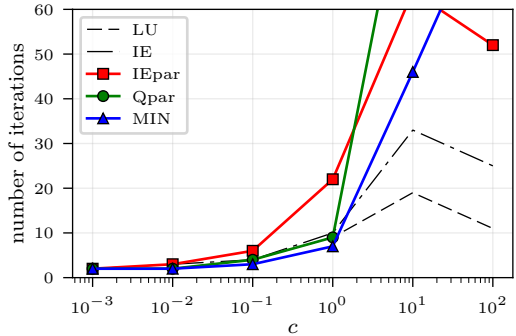}
    \caption{Advection equation}
    \label{fig:prec_adv}
  \end{subfigure}
  \par\medskip
  \begin{subfigure}[b]{\columnwidth}
    \centering
    \input{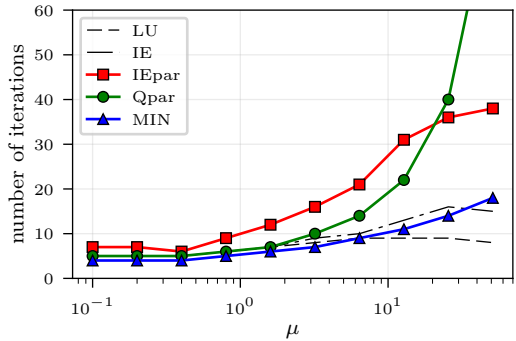}
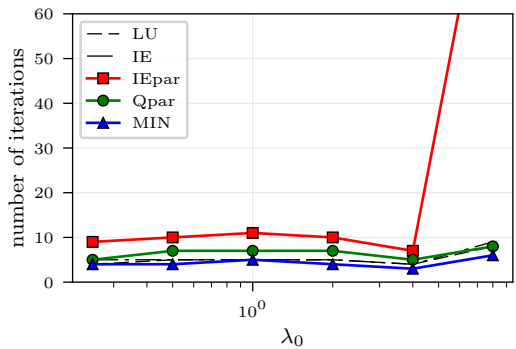
    \caption{Van der Pol oscillator}
    \label{fig:prec_vdp}
  \end{subfigure}
  \par\medskip 
  \begin{subfigure}[b]{\columnwidth}
    \centering
    \input{parallelSDC_preconditioner_fisher.pgf}
    \caption{Non-linear diffusion}
    \label{fig:prec_fisher}
  \end{subfigure} 
  \caption{Number of iterations for five different choices of $\QDmat$ for the four test problems A-D and varying parameters. The serial Gau\ss-Seidel-type matrices for the implicit Euler and the LU trick are shown as reference.}
  \label{fig:prec}
\end{figure}

\bigskip

The numerical comparison of the five different candidates for $\QDmat$, two leading to serial and three to parallel SDC iterations, shows no clear result.
The minimi\-zation-based approach $\QDmat^{\mathrm{MIN}}$ seems to be a reliable and sometimes even favorable choice, but only for the scenarios tested here.
Due to the intricate structure of $\Qmat$ there is no adequate mathematical theory to guide the selection of the $\QDmat$ matrix.
Even for a particular choice, estimating convergence speed with reasonable sharp bounds is not straightforward.
However, this section shows that in many cases parallel SDC iterations are indeed possible and can be implemented without any overhead.
Note that this strategy is only reasonable to use for shared-memory parallelization: in each iteration, the system~\eqref{eq:sdc_iteration} is solved in parallel, but in order to compute the right-hand side, each compute unit (core, threads etc.) needs to have data from all $M$ quadrature nodes.
If the collocation problem is large, e.g.\ due to a large amount of degrees-of-freedom $N$ in space, each unit has to receive the full vector $\vect{u}^k$, consisting of $MN$ variables.
If distributed-memory parallelization were used here, communication costs would be prohibitively high.

\section{Diagonalization of $\matr{Q}$}

While choosing a parallel preconditioner is a simple but rather heuristical idea, the second approach we describe here is more intricate.
In order to introduce parallelism across the nodes we now look at the diagonalization of the quadrature matrix $\Qmat$.
To this end, we write
\begin{align*}
  \Qmat = \matr{V}\matr{\Lambda}\matr{V}^{-1}
\end{align*}
where $\matr{\Lambda} = \mathrm{diag}(\lambda_1(\Qmat), ..., \lambda_M(\Qmat))^T$ is a diagonal matrix with eigenvalues $\lambda_i(\Qmat)\in\mathbb{C}$ of $\Qmat$ on the diagonal and $\matr{V}$ contains the eigenvectors of $\matr{Q}$.
For non-symmetric quadrature rules such as Gau\ss-Radau, this diagonalization of $\Qmat$ is possible, since all eigenvalues are distinct.

\bigskip

If $f$ is linear, i.e. if $f(u) = au$, then this leads to a parallel direct solver of the collocation problem~\eqref{eq:coll_prob}. We have
\begin{align}\label{eq:diag_linear}
  \left(\matr{I} - \dt\Qmat\vect{F}\right)(\vect{u}) &= \left(\matr{I} - a\dt\Qmat\right)\vect{u}\\
  &= \matr{V}\left(\matr{I} - a\dt\matr{\Lambda}\right)\matr{V}^{-1}
\end{align}
so that~\eqref{eq:coll_prob} can be solved in three simple steps:
\begin{enumerate}
  \item replace $\vect{u}_0$ by $\tilde{\vect{u}}_0 = \matr{V}^{-1}\vect{u}_0$ (serial)
  \item solve $\left(\matr{I} - a\dt\matr{\Lambda}\right)\tilde{\vect{u}} = \tilde{\vect{u}}_0$ (parallel in $M$)
  \item replace $\tilde{\vect{u}}$ by $\vect{u} = \matr{V}\tilde{\vect{u}}$ (serial)
\end{enumerate}
Since $\matr{\Lambda}$ is a diagonal matrix, step 2 can be done in parallel for all $M$ quadrature nodes at once.
Steps 1 and 3, in contrast, require global communication due to the structure of $\matr{V}$ and $\matr{V}^{-1}$ and as before, for problems with many degrees-of-freedom in space (i.e.~for $N$ large), this results in significant communication costs.
We therefore consider this strategy only suitable for shared-memory parallelization as well, which plays nicely with the rather small number $M$ of quadrature nodes used in typical applications.

\bigskip

Yet, the more severe restriction comes from the linearity of the right-hand side $f$. 
If $f$ is non-linear, we cannot follow~\eqref{eq:diag_linear} to solve the collocation problem in parallel.
This is due to the coupling of the non-linear $\vect{F}(\vect{u})$ with the matrix $\matr{Q}$, which does not allow the $\matr{V}$ to be extracted.
The obvious way to proceed is to linearize the problem using Newton's method.
We define
\begin{align*}
  \vect{G}(\vect{u}) = \vect{u} - \dt\Qmat\vect{F}(\vect{u}) - \vect{u}_0
\end{align*}
so that solving~\eqref{eq:coll_prob} is equivalent to finding a root of $\vect{G}(\vect{u})$.
For Newton iterations, we need the Jacobian $\vect{J}_{\vect{G}}$ of $\vect{G}$, which is given by
\begin{align*}
  \vect{J}_{\vect{G}}(\vect{u}) = \matr{I} - \dt\Qmat\vect{J}_{\vect{F}}(\vect{u})
\end{align*}
with
\begin{align}
  \vect{J}_{\vect{F}}(\vect{u}) = \mathrm{diag}\left(f'(u_1), ..., f'(u_M)\right)\in\mathbb{R}^{M\times N}.
\end{align}
Then, the Newton iteration is given by
\begin{align}
  \begin{split}\label{eq:newton_it}
    \vect{J}_{\vect{G}}(\vect{u}^k)\vect{e}^k &= -\vect{G}(\vect{u}^k),\\
    \vect{u}^{k+1} &= \vect{u}^k + \vect{e}^k.
  \end{split}
\end{align}
The matrix $\vect{J}_{\vect{G}}(\vect{u}^k)$ looks very much like the original operator of the collocation problem~\eqref{eq:coll_prob} for a linear function.
However, the matrix $\vect{J}_{\vect{F}}(\vect{u}^k)$ which appears in this problem still does not decouple from the matrix $\Qmat$, since for each quadrature node a different entry is given.
Again, $\matr{V}$ cannot be extracted and we need another step to obtain parallelism via diagonalization.

\bigskip

We note that for $k=0$ we have
\begin{align*}
  \vect{J}_{\vect{F}}(\vect{u}^0) &= \mathrm{diag}\left(f'(u^0_1), ..., f'(u^0_M)\right)\\
  &= \mathrm{diag}\left(f'(u^0_0), ..., f'(u^0_0)\right) = f'(u_0)\matr{I}_M
\end{align*}
if the iteration is started with $\vect{u}_0$ as initial guess.
This directly leads to a simplified Newton method with
\begin{align}
  \begin{split}\label{eq:simply_newton_it}
    \vect{J}_{\vect{G}}(\vect{u}^0)\vect{e}^k &= -\vect{G}(\vect{u}^k),\\
    \vect{u}^{k+1} &= \vect{u}^k + \vect{e}^k.
  \end{split}
\end{align}
where 
\begin{align}\label{eq:simply_Jg_def}
  \vect{J}_{\vect{G}}(\vect{u}^0) = \matr{I} - f'(u_0)\dt\Qmat
\end{align}
and here we indeed can use diagonalization of $\Qmat$ for parallelization across $M$ nodes.
For each iteration $k$ with a given iterate $\vect{u}^k$, the algorithm consists of these four steps:

\begin{enumerate}
  \item replace $\vect{r}^k = -\vect{G}(\vect{u}^k)$ by $\tilde{\vect{r}}^k = -\matr{V}^{-1}\vect{G}(\vect{u}^k)$ (serial)
  \item solve $\left(\matr{I} - f'(u_0)\dt\matr{\Lambda}\right)\tilde{\vect{e}}^k = \tilde{\vect{r}}^k$ (parallel in $M$)
  \item replace $\tilde{\vect{e}}^k$ by $\vect{e}^k = \matr{V}\tilde{\vect{e}}^k$ (serial)
  \item set $\vect{u}^{k+1} = \vect{u}^k + \vect{e}^k$ (parallel in $M$)
\end{enumerate}

We note that using simplified Newton methods for fully-implicit Runge-Kutta schemes like~\eqref{eq:coll_prob} is a standard way of solving these systems, see e.g. Section IV.8 in~\cite{Hairer_ODEsII}.

\bigskip

The price for using diagonalization of $\Qmat$ for parallelization is therefore the re-introduction of an iterative process as well as the need for the Jacobian of the right-hand side.
The latter, at least, has to be computed only once per time-step.
While for linear problems we obtain a parallel direct solver, a simplified Newton approach is required to obtain the same level of parallelism for non-linear problems.
The question now is, how much the approximation of the Jacobian $\vect{J}_{\vect{G}}(\vect{u}^k)$ by $\vect{J}_{\vect{G}}(\vect{u}^0)$ affects the convergence of the method and how this compares to standard SDC iterations.
In contrast to the previous section, we are now able to investigate this not only numerically but also on an analytic level.
It is well known that for suitable right-hand sides and initial guesses the standard, unmodified Newton method converges quadratically while the simplified Newton method only shows linear convergence, see e.g.~\cite{Kelley95,Jackson_NewtonRK}.
Yet, we are also interested in the constants and their dependence on the time-step size $\dt$.
More precisely, we can show the following result.

\begin{theorem}
  Let $f'$ be Lipschitz continuous with constant $\gamma_f$ and $\vect{u}_0, \vect{u}^k\in B(\dt) = \{\vect{u}\in\mathbb{R}^M: \lVert\vect{u} - \vect{u}^*\rVert\le c_1\dt\}$ for the exact solution $\vect{u}^*$ of the collocation problem~\eqref{eq:coll_prob} and the current iterate $\vect{u}^k$. 
  Furthermore, assume that $\vect{J}_{\vect{G}}(\vect{u}^*)$ is non-singular.
  Then the simplified Newton iteration~\eqref{eq:simply_newton_it} converges with
  \begin{align*}
    \lVert \vect{e}^{k+1}\rVert_\infty \le c\dt^2\lVert \vect{e}^k\rVert_\infty
  \end{align*}
  if $\dt$ is small enough.
\end{theorem}

\begin{proof}
  It is tempting to simply follow Theorem 5.4.2 in~\cite{Kelley95} for this proof: this states that the ``chord'' or simplified Newton method converges linearly to the exact solution $\vect{u}^*$, if the standard assumptions are satisfied (see 4.3 in~\cite{Kelley95}), namely if $\vect{u}^*$ is a solution of $\vect{G}(\vect{u})= 0$, $\vect{J}_{\vect{G}}$ is Lipschitz continuous and $\vect{J}_{\vect{G}}(\vect{u}^*)$ is non-singular.
  We first note that these standard assumptions are satisfied by the conditions we require here.
  In particular, $\vect{J}_{\vect{G}}$ is Lipschitz continuous with constant $\dt\gamma_f\lVert\Qmat\rVert_\infty$, because
  \begin{align}
    \begin{split}\label{eq:lipschitz_G}
      \lVert \vect{J}_{\vect{G}}(\vect{u}) - \vect{J}_{\vect{G}}(\vect{v})\rVert_\infty &\le \dt\lVert\Qmat\rVert_\infty \lVert\vect{J}_{\vect{F}}(\vect{u}) - \vect{J}_{\vect{F}}(\vect{v})\rVert_\infty\\
      &\le \dt\gamma_f\lVert\Qmat\rVert_\infty\lVert\vect{u}-\vect{v}\rVert_\infty.
    \end{split}
  \end{align}
  For the last inequality we note that
  \begin{align*}
    &\ \lVert\vect{J}_{\vect{F}}(\vect{u}) - \vect{J}_{\vect{F}}(\vect{v})\rVert_\infty\\
    =&\ \max_{m=1,...,M}\left\lvert\frac{\partial f}{\partial u_m}(u_m) - \frac{\partial f}{\partial u_m}(v_m)\right\rvert\\
    \le&\ \gamma_f \max_{m=1,...,M}\left\lvert u_m - v_m\right\rvert
    = \gamma_f \lVert\vect{u} - \vect{v}\rVert_\infty.
  \end{align*}
  However, using the estimate given in this theorem with the constants we have here would provide an estimate which is only linear in $\dt$ and therefore too pessimistic.
  To overcome this, we look at the more technical Theorem 5.4.1 in~\cite{Kelley95}, stating that for inaccurately computed $\vect{G}$ and $\vect{J}_{\vect{G}}$ the iteration error $\vect{e}^{k+1}$ can be estimated by a linear combination of the previous error $\vect{e}^k$, the error in the Jacobian $\vect{J}_{\vect{G}}$ as well as the error in the function $\vect{G}$.
  Here, the error in the Jacobian is simply the difference between $\vect{J}_{\vect{G}}(\vect{u}^k)$ and $\vect{J}_{\vect{G}}(\vect{u}^0)$ and we have with~\eqref{eq:lipschitz_G}
  \begin{align}
    \begin{split}\label{eq:est_jg0-jgk}
    \left\lVert\vect{J}_{\vect{G}}(\vect{u}^0) - \vect{J}_{\vect{G}}(\vect{u}^k)\right\rVert_\infty &\le \dt\gamma_f\lVert\Qmat\rVert_\infty\lVert\vect{u}_0-\vect{u}^k\rVert_\infty\\
    &\le c_1\dt^2,  
    \end{split}
  \end{align}
  since $\vect{u}^k$ and $\vect{u}^0$ are both in $B(\dt)$.
  The function $\vect{G}$ is evaluated exactly, so that the error in $\vect{G}$ is just zero.
  Then, looking at the proof of Theorem 5.4.1, the last estimate reads in our notation
  \begin{align}\label{eq:est_proof_kelley}
    \lVert\vect{e}^{k+1}\rVert_\infty \le \gamma_f\dt\left\lVert\vect{e}^k\right\rVert^2_\infty + 16P\left\lVert\vect{e}^k\right\rVert_\infty + 0.
  \end{align}
  for 
  \begin{align*}
    P = \left\lVert\vect{J}_{\vect{G}}(\vect{u}^*)^{-1}\right\rVert_\infty^2&\left\lVert\vect{J}_{\vect{G}}(\vect{u}^*)\right\rVert_\infty\cdot\\
    &\left\lVert\vect{J}_{\vect{G}}(\vect{u}^0) - \vect{J}_{\vect{G}}(\vect{u}^k)\right\rVert_\infty.
  \end{align*}
  For $\dt$ small enough and by assumption, we can bound both $\left\lVert\vect{J}_{\vect{G}}(\vect{u}^*)^{-1}\right\rVert_\infty^2$ and $\left\lVert\vect{J}_{\vect{G}}(\vect{u}^*)\right\rVert_\infty$ by some constant $c_2$.
  Also, we note that 
  \begin{align*}
    \left\lVert\vect{e}^k\right\rVert^2_\infty \le \dt\left\lVert\vect{e}^k\right\rVert_\infty.
  \end{align*}
  Then, putting all the results together Inequality~\eqref{eq:est_proof_kelley} yields
  \begin{align*}
    \lVert \vect{e}^{k+1}\rVert_\infty &\le \gamma_f\dt^2\left\lVert\vect{e}^k\right\rVert_\infty + 16c_1c_2\dt^2\left\lVert\vect{e}^k\right\rVert_\infty\\
    &= c\dt^2\left\lVert\vect{e}^k\right\rVert_\infty
  \end{align*}
  which concludes the proof.\hfill\* \qed

\end{proof}

This shows that the simplified Newton iteration converges linearly with a contraction factor of the order of $\mathcal{O}(\dt^2)$. 

\bigskip

An objection one might raise is that the eigenvalues of $\Qmat$ and thus the entries of $\matr{\Lambda}$ are complex, making implementations slightly more cumbersome.
To avoid this, we can ``borrow'' the preconditioning idea from SDC, i.e.~instead of using the simplified Newton method with Eq.~\eqref{eq:simply_Jg_def}, we use
\begin{align*}
  \vect{J}^\Delta_{\vect{G}}(\vect{u}^0) = \matr{I} - f'(u_0)\dt\QDmat
\end{align*}
so that the simplified Newton method becomes an inexact simplified Newton method:
\begin{align}
  \begin{split}\label{eq:isimply_newton_it}
    \vect{J}^\Delta_{\vect{G}}(\vect{u}^0)\vect{e}^k &= -\vect{G}(\vect{u}^k),\\
    \vect{u}^{k+1} &= \vect{u}^k + \vect{e}^k.
  \end{split}
\end{align}

\begin{theorem}
  Let $f'$ be Lipschitz continuous with constant $\gamma_f$ and $\vect{u}_0, \vect{u}^k\in B(\dt) = \{\vect{u}\in\mathbb{R}^M: \lVert\vect{u} - \vect{u}^*\rVert\le c_1\dt\}$ for the exact solution $\vect{u}^*$ of the collocation problem~\eqref{eq:coll_prob}. 
  Furthermore, assume that $\vect{J}_{\vect{G}}(\vect{u}^*)$ and $\vect{J}^\Delta_{\vect{G}}(\vect{u}^0)$ are non-singular.
  Then the inexact simplified Newton iteration~\eqref{eq:isimply_newton_it} converges with
  \begin{align*}
    \lVert \vect{e}^{k+1}\rVert_\infty \le c\dt\lVert \vect{e}^k\rVert_\infty
  \end{align*}
  if $\dt$ is small enough.
\end{theorem}

\begin{proof}
  We use again the final estimate of the proof of Theorem 5.4.1 in~\cite{Kelley95}, which for
  \begin{align*}
    \tilde{P} = \left\lVert\vect{J}_{\vect{G}}(\vect{u}^*)^{-1}\right\rVert_\infty^2\left\lVert\vect{J}_{\vect{G}}(\vect{u}^*)\right\rVert_\infty\left\lVert\vect{J}_{\vect{G}}^\Delta(\vect{u}^0) - \vect{J}_{\vect{G}}(\vect{u}^k)\right\rVert_\infty
  \end{align*}
  now reads
  \begin{align*}
    \lVert \vect{e}^{k+1}\rVert_\infty \le \gamma_f\dt\left\lVert\vect{e}^k\right\rVert^2_\infty + 16\tilde{P}\left\lVert\vect{e}^k\right\rVert_\infty + 0,
  \end{align*}  
  i.e.~we simply replaced $\vect{J}_{\vect{G}}(\vect{u}_0)$ by $\vect{J}_{\vect{G}}^\Delta(\vect{u}^0)$ in~\eqref{eq:est_proof_kelley}.
  Then, we have 
  \begin{align*}
    \left\lVert\vect{J}_{\vect{G}}^\Delta(\vect{u}^0) - \vect{J}_{\vect{G}}(\vect{u}^k)\right\rVert_\infty \le \ &\left\lVert\vect{J}_{\vect{G}}^\Delta(\vect{u}^0) - \vect{J}_{\vect{G}}(\vect{u}^0)\right\rVert_\infty \\
    &+ \left\lVert\vect{J}_{\vect{G}}(\vect{u}^0) - \vect{J}_{\vect{G}}(\vect{u}^k)\right\rVert_\infty
  \end{align*}
  and we note that the second term is the same we had in~\eqref{eq:est_jg0-jgk}.
  For the first term it is
  \begin{align*}
    \left\lVert\vect{J}_{\vect{G}}^\Delta(\vect{u}^0) - \vect{J}_{\vect{G}}(\vect{u}^0)\right\rVert_\infty \le \dt\lvert f'(u_0)\rvert \left\lVert\Qmat-\QDmat\right\rVert_\infty
  \end{align*}
  so that
  \begin{align*}
    \left\lVert\vect{J}_{\vect{G}}^\Delta(\vect{u}^0) - \vect{J}_{\vect{G}}(\vect{u}^k)\right\rVert_\infty \le\ &\dt\lvert f'(u_0)\rvert \left\lVert\Qmat-\QDmat\right\rVert_\infty\\
    &+ c_1\dt^2,
  \end{align*}
  see~\eqref{eq:est_jg0-jgk}.
  Therefore, we obtain
  \begin{align*}
    \lVert \vect{e}^{k+1}\rVert_\infty \le &\gamma_f\dt^2\left\lVert\vect{e}^k\right\rVert_\infty + 16c_1c_2\dt^2\left\lVert\vect{e}^k\right\rVert_\infty\\
    &+ 16c_2c_3\dt\left\lVert\Qmat-\QDmat\right\rVert_\infty\left\lVert\vect{e}^k\right\rVert_\infty
  \end{align*}
  so that in summary
  \begin{align*}
    \lVert \vect{e}^{k+1}\rVert_\infty \le c\dt\left\lVert\vect{e}^k\right\rVert_\infty
  \end{align*}
  which concludes the proof.\hfill\* \qed
\end{proof}

Aside from constants, this is the same rate as the classical SDC convergence rate~\cite{Tang2013,RuprechtSpeck2016}.
However, we can see in the last estimate of this proof that with the introduction of $\QDmat$ the order of convergence in $\dt$ becomes linear only because of the factor $\left\lVert\Qmat-\QDmat\right\rVert_\infty$ which turns out to be small in many cases.

\bigskip

We test both theorems using the the nonlinear diffusion equation, see also Problem D in the last section:
\begin{align}\label{eq:fisher}
  \begin{split}
    u_t &= u_{xx} + \lambda^2_0 u(1-u^\nu)\quad \text{on } \mathbb{R}\times[0,T],\\
    u(x,0) &= \left(1 + (2^{\nu/2}-1)e^{-(\nu/2)\delta x}\right)^{-\frac{2}{\nu}}
  \end{split}
\end{align}
for constants $\delta>0$, $\lambda_0>0$ and $\nu\in\mathbb{N}$.
For particular choices of $\delta$ and $\lambda_0$ this problem has an exact solution which we can use to evaluate the error. 
Note that instead of having boundaries at $\pm\infty$ we use the exact solution at the boundaries on a fixed interval $[a,b] = [-5,5]$.
We use $M=5$ Gau\ss-Radau nodes with the LU trick of~\cite{Weiser2014} as preconditioner $\QDmat$, $T=0.1$ with $2,...,16$ time-steps, finite differences with $N=2047$ degrees-of-freedom in space and $\nu=1$, $\lambda_0 = 5$ and $\delta$ given by the relation in~\cite{Feng2008481}.
For SDC, the implicit systems at each node are solved using a standard (spatial) Newton method with tolerance $10^{-12}$ while the linear systems for the Newton-like approaches are solved directly.
In order to compute the rate of error reduction, we compute the ratio between the error on all quadrature nodes before and after iteration $2$ at the last time-step. 
For SDC, the simplified Newton method as well as the inexact Newton method, these ratios are shown in Figure~\ref{fig:fisher} for different time-step sizes $\dt$.

\begin{figure}[t]
  \centering
  \input{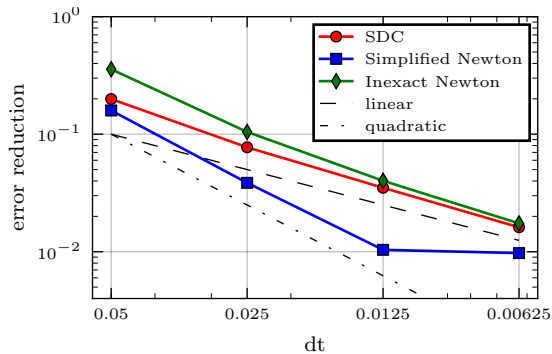}
  \caption{Convergence rates of SDC, simplified Newton and inexact Newton for the nonlinear diffusion problem~\eqref{eq:fisher}. Difference between the error at iteration $2$ and $3$ for the last time-step is shown.}
  \label{fig:fisher}
\end{figure}
 
We can nicely observe linear convergence of SDC as well as the quadratic convergence of the simplified Newton method. 
The inexact Newton method converges slightly faster than linear, but is still far away from quadratic convergence.
It has the largest ratio of all three methods, suggesting that the constants are higher than for the other methods.
We note that in this case we have $\left\lVert\Qmat-\QDmat\right\rVert_\infty \approx 0.265$.
We also observe that the simplified Newton method does not only have better convergence rates but also has the smallest ratios of all methods for all $\dt$.
Thus, if complex arithmetic is not a problem (or circumvented in other ways) and if the Jacobian of the right-hand side is available, this approach seems to be preferable.

\section{Conclusion and outlook}

In this paper we have introduced, discussed and analyzed two different approaches for parallelizing iterations of implicit spectral deferred corrections across the quadrature nodes.
While parallel-in-time algorithms like PFASST allow to compute multiple time-steps simultaneously and target large-scale parallelism in time, the ideas presented here focus on the small-scale parallelization of a single time-step. 
In the sense of~\cite{Burrage1997}, these ideas therefore realize ``parallelization across the method'' for SDC.

\bigskip

The first approach allows simultaneous evaluation of SDC updates on all quadrature nodes by using a diagonal preconditioning matrix $\QDmat$ instead of standard Gau\ss-Seidel-like choices.
With $\QDmat$ being diagonal, all $M$ ``stages'' of SDC, i.e.~updates for the solution at each quadrature node, can be computed using $M$ processes. 
This includes solving an implicit system at the nodes as well as evaluation of the right-hand side of the initial value problem, all of which can now be done in parallel.
We chose three different diagonal matrices and analyzed their impact on the convergence of SDC using four different test problems.
While for non-stiff cases all candidates performed rather well, only the third alternative using a minimization approach was able to work about as good as the standard SDC preconditioners for stiff problems, too.
Yet, so far no conclusive theory exists to estimate the impact of the choice of $\QDmat$ in the convergence of SDC, let alone for the derivation of optimal (serial or parallel) preconditioners.
For this approach, generic and rather obvious choices of $\QDmat$ were considered and it can be expected that better candidates might exists, in particular if the problem at hand is taken into account.

\bigskip

The second approach uses diagonalization of the quadrature matrix $\Qmat$ to achieve parallelism across the nodes.
For linear problems and suitable choices of quad\-rature rules, this yields a direct parallel solver of the collocation problem.
For non-linear problem, though, this is not applicable and linearization via Newton's method is needed. 
The resulting linear systems for each Newton iteration looks similar to a linear collocation problem, but only when the full Jacobian is frozen at the first node the diagonalization technique is applicable.
This simplified Newton method shows quite remarkable convergence properties when compared to standard SDC. 
Yet, complex arithmetic is needed due to the complex eigenvalues of $\Qmat$.
If this is an issue, the simplified Newton method can be further extended to an inexact simplified Newton method by diagonalizing $\QDmat$ instead of $\Qmat$.
For these two methods, simplified and inexact simplified Newton, we were able to prove linear convergence and show that the constants depend quadratically in the first and linearly in the second case on $\dt$.
This makes the convergence of the inexact method about as fast as standard SDC, which we could also observe numerically using a non-linear test problem.

\bigskip

Both approaches are rather easy to implement and the code used for generating the numerical results of this paper can be found online within the \texttt{pySDC} framework~\cite{robert_speck_2017_376982}.
We firmly believe that there are more ways to achieve parallelism within SDC, either across the nodes or even across the iterations.
One possibility is to apply methods for parallelization across the step like Parareal~\cite{LionsEtAl2001} for the preconditioner itself.
When choosing the implicit Euler for $\QDmat$, each iteration of SDC is just a sequence of implicit Euler steps with a modified right-hand side.
Thus, if a method is able to parallelize $M$ implicit Euler steps, we could use it within SDC as parallel preconditioner.
If this method is an iterative method itself, then with SDC being the outer iteration ideas like inexact SDC~\cite{SpeckEtAl2016} could further speed up the algorithm.

\bibliographystyle{spmpsci}
\bibliography{refs}

\end{document}

%% file: parallelSDC_preconditioner_fisher.pgf
%% Creator: Matplotlib, PGF backend
%%
%% To include the figure in your LaTeX document, write
%%   \input{<filename>.pgf}
%%
%% Make sure the required packages are loaded in your preamble
%%   \usepackage{pgf}
%%
%% Figures using additional raster images can only be included by \input if
%% they are in the same directory as the main LaTeX file. For loading figures
%% from other directories you can use the `import` package
%% and then include the figures with
%%   \import{<path to file>}{<filename>.pgf}
%%
%% Matplotlib used the following preamble
%%   \usepackage[utf8x]{inputenc}
%%   \usepackage[T1]{fontenc}
%%
\begingroup%
\makeatletter%
\begin{pgfpicture}%
\pgfpathrectangle{\pgfpointorigin}{\pgfqpoint{2.688404in}{1.792995in}}%
\pgfusepath{use as bounding box, clip}%
\begin{pgfscope}%
\pgfsetbuttcap%
\pgfsetmiterjoin%
\definecolor{currentfill}{rgb}{1.000000,1.000000,1.000000}%
\pgfsetfillcolor{currentfill}%
\pgfsetlinewidth{0.000000pt}%
\definecolor{currentstroke}{rgb}{1.000000,1.000000,1.000000}%
\pgfsetstrokecolor{currentstroke}%
\pgfsetdash{}{0pt}%
\pgfpathmoveto{\pgfqpoint{0.000000in}{0.000000in}}%
\pgfpathlineto{\pgfqpoint{2.688404in}{0.000000in}}%
\pgfpathlineto{\pgfqpoint{2.688404in}{1.792995in}}%
\pgfpathlineto{\pgfqpoint{0.000000in}{1.792995in}}%
\pgfpathclose%
\pgfusepath{fill}%
\end{pgfscope}%
\begin{pgfscope}%
\pgfsetbuttcap%
\pgfsetmiterjoin%
\definecolor{currentfill}{rgb}{1.000000,1.000000,1.000000}%
\pgfsetfillcolor{currentfill}%
\pgfsetlinewidth{0.000000pt}%
\definecolor{currentstroke}{rgb}{0.000000,0.000000,0.000000}%
\pgfsetstrokecolor{currentstroke}%
\pgfsetstrokeopacity{0.000000}%
\pgfsetdash{}{0pt}%
\pgfpathmoveto{\pgfqpoint{0.362752in}{0.347569in}}%
\pgfpathlineto{\pgfqpoint{2.643404in}{0.347569in}}%
\pgfpathlineto{\pgfqpoint{2.643404in}{1.747995in}}%
\pgfpathlineto{\pgfqpoint{0.362752in}{1.747995in}}%
\pgfpathclose%
\pgfusepath{fill}%
\end{pgfscope}%
\begin{pgfscope}%
\pgfpathrectangle{\pgfqpoint{0.362752in}{0.347569in}}{\pgfqpoint{2.280651in}{1.400426in}} %
\pgfusepath{clip}%
\pgfsetrectcap%
\pgfsetroundjoin%
\pgfsetlinewidth{0.501875pt}%
\definecolor{currentstroke}{rgb}{0.690196,0.690196,0.690196}%
\pgfsetstrokecolor{currentstroke}%
\pgfsetstrokeopacity{0.250000}%
\pgfsetdash{}{0pt}%
\pgfpathmoveto{\pgfqpoint{1.295746in}{0.347569in}}%
\pgfpathlineto{\pgfqpoint{1.295746in}{1.747995in}}%
\pgfusepath{stroke}%
\end{pgfscope}%
\begin{pgfscope}%
\pgfsetbuttcap%
\pgfsetroundjoin%
\definecolor{currentfill}{rgb}{0.000000,0.000000,0.000000}%
\pgfsetfillcolor{currentfill}%
\pgfsetlinewidth{0.501875pt}%
\definecolor{currentstroke}{rgb}{0.000000,0.000000,0.000000}%
\pgfsetstrokecolor{currentstroke}%
\pgfsetdash{}{0pt}%
\pgfsys@defobject{currentmarker}{\pgfqpoint{0.000000in}{-0.048611in}}{\pgfqpoint{0.000000in}{0.000000in}}{%
\pgfpathmoveto{\pgfqpoint{0.000000in}{0.000000in}}%
\pgfpathlineto{\pgfqpoint{0.000000in}{-0.048611in}}%
\pgfusepath{stroke,fill}%
}%
\begin{pgfscope}%
\pgfsys@transformshift{1.295746in}{0.347569in}%
\pgfsys@useobject{currentmarker}{}%
\end{pgfscope}%
\end{pgfscope}%
\begin{pgfscope}%
\pgftext[x=1.295746in,y=0.250347in,,top]{\rmfamily\fontsize{6.000000}{7.200000}\selectfont \(\displaystyle 10^{0}\)}%
\end{pgfscope}%
\begin{pgfscope}%
\pgfsetbuttcap%
\pgfsetroundjoin%
\definecolor{currentfill}{rgb}{0.000000,0.000000,0.000000}%
\pgfsetfillcolor{currentfill}%
\pgfsetlinewidth{0.250937pt}%
\definecolor{currentstroke}{rgb}{0.000000,0.000000,0.000000}%
\pgfsetstrokecolor{currentstroke}%
\pgfsetdash{}{0pt}%
\pgfsys@defobject{currentmarker}{\pgfqpoint{0.000000in}{-0.027778in}}{\pgfqpoint{0.000000in}{0.000000in}}{%
\pgfpathmoveto{\pgfqpoint{0.000000in}{0.000000in}}%
\pgfpathlineto{\pgfqpoint{0.000000in}{-0.027778in}}%
\pgfusepath{stroke,fill}%
}%
\begin{pgfscope}%
\pgfsys@transformshift{0.575489in}{0.347569in}%
\pgfsys@useobject{currentmarker}{}%
\end{pgfscope}%
\end{pgfscope}%
\begin{pgfscope}%
\pgfsetbuttcap%
\pgfsetroundjoin%
\definecolor{currentfill}{rgb}{0.000000,0.000000,0.000000}%
\pgfsetfillcolor{currentfill}%
\pgfsetlinewidth{0.250937pt}%
\definecolor{currentstroke}{rgb}{0.000000,0.000000,0.000000}%
\pgfsetstrokecolor{currentstroke}%
\pgfsetdash{}{0pt}%
\pgfsys@defobject{currentmarker}{\pgfqpoint{0.000000in}{-0.027778in}}{\pgfqpoint{0.000000in}{0.000000in}}{%
\pgfpathmoveto{\pgfqpoint{0.000000in}{0.000000in}}%
\pgfpathlineto{\pgfqpoint{0.000000in}{-0.027778in}}%
\pgfusepath{stroke,fill}%
}%
\begin{pgfscope}%
\pgfsys@transformshift{0.747590in}{0.347569in}%
\pgfsys@useobject{currentmarker}{}%
\end{pgfscope}%
\end{pgfscope}%
\begin{pgfscope}%
\pgfsetbuttcap%
\pgfsetroundjoin%
\definecolor{currentfill}{rgb}{0.000000,0.000000,0.000000}%
\pgfsetfillcolor{currentfill}%
\pgfsetlinewidth{0.250937pt}%
\definecolor{currentstroke}{rgb}{0.000000,0.000000,0.000000}%
\pgfsetstrokecolor{currentstroke}%
\pgfsetdash{}{0pt}%
\pgfsys@defobject{currentmarker}{\pgfqpoint{0.000000in}{-0.027778in}}{\pgfqpoint{0.000000in}{0.000000in}}{%
\pgfpathmoveto{\pgfqpoint{0.000000in}{0.000000in}}%
\pgfpathlineto{\pgfqpoint{0.000000in}{-0.027778in}}%
\pgfusepath{stroke,fill}%
}%
\begin{pgfscope}%
\pgfsys@transformshift{0.881082in}{0.347569in}%
\pgfsys@useobject{currentmarker}{}%
\end{pgfscope}%
\end{pgfscope}%
\begin{pgfscope}%
\pgfsetbuttcap%
\pgfsetroundjoin%
\definecolor{currentfill}{rgb}{0.000000,0.000000,0.000000}%
\pgfsetfillcolor{currentfill}%
\pgfsetlinewidth{0.250937pt}%
\definecolor{currentstroke}{rgb}{0.000000,0.000000,0.000000}%
\pgfsetstrokecolor{currentstroke}%
\pgfsetdash{}{0pt}%
\pgfsys@defobject{currentmarker}{\pgfqpoint{0.000000in}{-0.027778in}}{\pgfqpoint{0.000000in}{0.000000in}}{%
\pgfpathmoveto{\pgfqpoint{0.000000in}{0.000000in}}%
\pgfpathlineto{\pgfqpoint{0.000000in}{-0.027778in}}%
\pgfusepath{stroke,fill}%
}%
\begin{pgfscope}%
\pgfsys@transformshift{0.990153in}{0.347569in}%
\pgfsys@useobject{currentmarker}{}%
\end{pgfscope}%
\end{pgfscope}%
\begin{pgfscope}%
\pgfsetbuttcap%
\pgfsetroundjoin%
\definecolor{currentfill}{rgb}{0.000000,0.000000,0.000000}%
\pgfsetfillcolor{currentfill}%
\pgfsetlinewidth{0.250937pt}%
\definecolor{currentstroke}{rgb}{0.000000,0.000000,0.000000}%
\pgfsetstrokecolor{currentstroke}%
\pgfsetdash{}{0pt}%
\pgfsys@defobject{currentmarker}{\pgfqpoint{0.000000in}{-0.027778in}}{\pgfqpoint{0.000000in}{0.000000in}}{%
\pgfpathmoveto{\pgfqpoint{0.000000in}{0.000000in}}%
\pgfpathlineto{\pgfqpoint{0.000000in}{-0.027778in}}%
\pgfusepath{stroke,fill}%
}%
\begin{pgfscope}%
\pgfsys@transformshift{1.082371in}{0.347569in}%
\pgfsys@useobject{currentmarker}{}%
\end{pgfscope}%
\end{pgfscope}%
\begin{pgfscope}%
\pgfsetbuttcap%
\pgfsetroundjoin%
\definecolor{currentfill}{rgb}{0.000000,0.000000,0.000000}%
\pgfsetfillcolor{currentfill}%
\pgfsetlinewidth{0.250937pt}%
\definecolor{currentstroke}{rgb}{0.000000,0.000000,0.000000}%
\pgfsetstrokecolor{currentstroke}%
\pgfsetdash{}{0pt}%
\pgfsys@defobject{currentmarker}{\pgfqpoint{0.000000in}{-0.027778in}}{\pgfqpoint{0.000000in}{0.000000in}}{%
\pgfpathmoveto{\pgfqpoint{0.000000in}{0.000000in}}%
\pgfpathlineto{\pgfqpoint{0.000000in}{-0.027778in}}%
\pgfusepath{stroke,fill}%
}%
\begin{pgfscope}%
\pgfsys@transformshift{1.162254in}{0.347569in}%
\pgfsys@useobject{currentmarker}{}%
\end{pgfscope}%
\end{pgfscope}%
\begin{pgfscope}%
\pgfsetbuttcap%
\pgfsetroundjoin%
\definecolor{currentfill}{rgb}{0.000000,0.000000,0.000000}%
\pgfsetfillcolor{currentfill}%
\pgfsetlinewidth{0.250937pt}%
\definecolor{currentstroke}{rgb}{0.000000,0.000000,0.000000}%
\pgfsetstrokecolor{currentstroke}%
\pgfsetdash{}{0pt}%
\pgfsys@defobject{currentmarker}{\pgfqpoint{0.000000in}{-0.027778in}}{\pgfqpoint{0.000000in}{0.000000in}}{%
\pgfpathmoveto{\pgfqpoint{0.000000in}{0.000000in}}%
\pgfpathlineto{\pgfqpoint{0.000000in}{-0.027778in}}%
\pgfusepath{stroke,fill}%
}%
\begin{pgfscope}%
\pgfsys@transformshift{1.232716in}{0.347569in}%
\pgfsys@useobject{currentmarker}{}%
\end{pgfscope}%
\end{pgfscope}%
\begin{pgfscope}%
\pgfsetbuttcap%
\pgfsetroundjoin%
\definecolor{currentfill}{rgb}{0.000000,0.000000,0.000000}%
\pgfsetfillcolor{currentfill}%
\pgfsetlinewidth{0.250937pt}%
\definecolor{currentstroke}{rgb}{0.000000,0.000000,0.000000}%
\pgfsetstrokecolor{currentstroke}%
\pgfsetdash{}{0pt}%
\pgfsys@defobject{currentmarker}{\pgfqpoint{0.000000in}{-0.027778in}}{\pgfqpoint{0.000000in}{0.000000in}}{%
\pgfpathmoveto{\pgfqpoint{0.000000in}{0.000000in}}%
\pgfpathlineto{\pgfqpoint{0.000000in}{-0.027778in}}%
\pgfusepath{stroke,fill}%
}%
\begin{pgfscope}%
\pgfsys@transformshift{1.710410in}{0.347569in}%
\pgfsys@useobject{currentmarker}{}%
\end{pgfscope}%
\end{pgfscope}%
\begin{pgfscope}%
\pgfsetbuttcap%
\pgfsetroundjoin%
\definecolor{currentfill}{rgb}{0.000000,0.000000,0.000000}%
\pgfsetfillcolor{currentfill}%
\pgfsetlinewidth{0.250937pt}%
\definecolor{currentstroke}{rgb}{0.000000,0.000000,0.000000}%
\pgfsetstrokecolor{currentstroke}%
\pgfsetdash{}{0pt}%
\pgfsys@defobject{currentmarker}{\pgfqpoint{0.000000in}{-0.027778in}}{\pgfqpoint{0.000000in}{0.000000in}}{%
\pgfpathmoveto{\pgfqpoint{0.000000in}{0.000000in}}%
\pgfpathlineto{\pgfqpoint{0.000000in}{-0.027778in}}%
\pgfusepath{stroke,fill}%
}%
\begin{pgfscope}%
\pgfsys@transformshift{1.952973in}{0.347569in}%
\pgfsys@useobject{currentmarker}{}%
\end{pgfscope}%
\end{pgfscope}%
\begin{pgfscope}%
\pgfsetbuttcap%
\pgfsetroundjoin%
\definecolor{currentfill}{rgb}{0.000000,0.000000,0.000000}%
\pgfsetfillcolor{currentfill}%
\pgfsetlinewidth{0.250937pt}%
\definecolor{currentstroke}{rgb}{0.000000,0.000000,0.000000}%
\pgfsetstrokecolor{currentstroke}%
\pgfsetdash{}{0pt}%
\pgfsys@defobject{currentmarker}{\pgfqpoint{0.000000in}{-0.027778in}}{\pgfqpoint{0.000000in}{0.000000in}}{%
\pgfpathmoveto{\pgfqpoint{0.000000in}{0.000000in}}%
\pgfpathlineto{\pgfqpoint{0.000000in}{-0.027778in}}%
\pgfusepath{stroke,fill}%
}%
\begin{pgfscope}%
\pgfsys@transformshift{2.125074in}{0.347569in}%
\pgfsys@useobject{currentmarker}{}%
\end{pgfscope}%
\end{pgfscope}%
\begin{pgfscope}%
\pgfsetbuttcap%
\pgfsetroundjoin%
\definecolor{currentfill}{rgb}{0.000000,0.000000,0.000000}%
\pgfsetfillcolor{currentfill}%
\pgfsetlinewidth{0.250937pt}%
\definecolor{currentstroke}{rgb}{0.000000,0.000000,0.000000}%
\pgfsetstrokecolor{currentstroke}%
\pgfsetdash{}{0pt}%
\pgfsys@defobject{currentmarker}{\pgfqpoint{0.000000in}{-0.027778in}}{\pgfqpoint{0.000000in}{0.000000in}}{%
\pgfpathmoveto{\pgfqpoint{0.000000in}{0.000000in}}%
\pgfpathlineto{\pgfqpoint{0.000000in}{-0.027778in}}%
\pgfusepath{stroke,fill}%
}%
\begin{pgfscope}%
\pgfsys@transformshift{2.258566in}{0.347569in}%
\pgfsys@useobject{currentmarker}{}%
\end{pgfscope}%
\end{pgfscope}%
\begin{pgfscope}%
\pgfsetbuttcap%
\pgfsetroundjoin%
\definecolor{currentfill}{rgb}{0.000000,0.000000,0.000000}%
\pgfsetfillcolor{currentfill}%
\pgfsetlinewidth{0.250937pt}%
\definecolor{currentstroke}{rgb}{0.000000,0.000000,0.000000}%
\pgfsetstrokecolor{currentstroke}%
\pgfsetdash{}{0pt}%
\pgfsys@defobject{currentmarker}{\pgfqpoint{0.000000in}{-0.027778in}}{\pgfqpoint{0.000000in}{0.000000in}}{%
\pgfpathmoveto{\pgfqpoint{0.000000in}{0.000000in}}%
\pgfpathlineto{\pgfqpoint{0.000000in}{-0.027778in}}%
\pgfusepath{stroke,fill}%
}%
\begin{pgfscope}%
\pgfsys@transformshift{2.367637in}{0.347569in}%
\pgfsys@useobject{currentmarker}{}%
\end{pgfscope}%
\end{pgfscope}%
\begin{pgfscope}%
\pgfsetbuttcap%
\pgfsetroundjoin%
\definecolor{currentfill}{rgb}{0.000000,0.000000,0.000000}%
\pgfsetfillcolor{currentfill}%
\pgfsetlinewidth{0.250937pt}%
\definecolor{currentstroke}{rgb}{0.000000,0.000000,0.000000}%
\pgfsetstrokecolor{currentstroke}%
\pgfsetdash{}{0pt}%
\pgfsys@defobject{currentmarker}{\pgfqpoint{0.000000in}{-0.027778in}}{\pgfqpoint{0.000000in}{0.000000in}}{%
\pgfpathmoveto{\pgfqpoint{0.000000in}{0.000000in}}%
\pgfpathlineto{\pgfqpoint{0.000000in}{-0.027778in}}%
\pgfusepath{stroke,fill}%
}%
\begin{pgfscope}%
\pgfsys@transformshift{2.459855in}{0.347569in}%
\pgfsys@useobject{currentmarker}{}%
\end{pgfscope}%
\end{pgfscope}%
\begin{pgfscope}%
\pgfsetbuttcap%
\pgfsetroundjoin%
\definecolor{currentfill}{rgb}{0.000000,0.000000,0.000000}%
\pgfsetfillcolor{currentfill}%
\pgfsetlinewidth{0.250937pt}%
\definecolor{currentstroke}{rgb}{0.000000,0.000000,0.000000}%
\pgfsetstrokecolor{currentstroke}%
\pgfsetdash{}{0pt}%
\pgfsys@defobject{currentmarker}{\pgfqpoint{0.000000in}{-0.027778in}}{\pgfqpoint{0.000000in}{0.000000in}}{%
\pgfpathmoveto{\pgfqpoint{0.000000in}{0.000000in}}%
\pgfpathlineto{\pgfqpoint{0.000000in}{-0.027778in}}%
\pgfusepath{stroke,fill}%
}%
\begin{pgfscope}%
\pgfsys@transformshift{2.539738in}{0.347569in}%
\pgfsys@useobject{currentmarker}{}%
\end{pgfscope}%
\end{pgfscope}%
\begin{pgfscope}%
\pgfsetbuttcap%
\pgfsetroundjoin%
\definecolor{currentfill}{rgb}{0.000000,0.000000,0.000000}%
\pgfsetfillcolor{currentfill}%
\pgfsetlinewidth{0.250937pt}%
\definecolor{currentstroke}{rgb}{0.000000,0.000000,0.000000}%
\pgfsetstrokecolor{currentstroke}%
\pgfsetdash{}{0pt}%
\pgfsys@defobject{currentmarker}{\pgfqpoint{0.000000in}{-0.027778in}}{\pgfqpoint{0.000000in}{0.000000in}}{%
\pgfpathmoveto{\pgfqpoint{0.000000in}{0.000000in}}%
\pgfpathlineto{\pgfqpoint{0.000000in}{-0.027778in}}%
\pgfusepath{stroke,fill}%
}%
\begin{pgfscope}%
\pgfsys@transformshift{2.610199in}{0.347569in}%
\pgfsys@useobject{currentmarker}{}%
\end{pgfscope}%
\end{pgfscope}%
\begin{pgfscope}%
\pgftext[x=1.503078in,y=0.108124in,,top]{\rmfamily\fontsize{8.000000}{9.600000}\selectfont \(\displaystyle \lambda_0\)}%
\end{pgfscope}%
\begin{pgfscope}%
\pgfpathrectangle{\pgfqpoint{0.362752in}{0.347569in}}{\pgfqpoint{2.280651in}{1.400426in}} %
\pgfusepath{clip}%
\pgfsetrectcap%
\pgfsetroundjoin%
\pgfsetlinewidth{0.501875pt}%
\definecolor{currentstroke}{rgb}{0.690196,0.690196,0.690196}%
\pgfsetstrokecolor{currentstroke}%
\pgfsetstrokeopacity{0.250000}%
\pgfsetdash{}{0pt}%
\pgfpathmoveto{\pgfqpoint{0.362752in}{0.347569in}}%
\pgfpathlineto{\pgfqpoint{2.643404in}{0.347569in}}%
\pgfusepath{stroke}%
\end{pgfscope}%
\begin{pgfscope}%
\pgfsetbuttcap%
\pgfsetroundjoin%
\definecolor{currentfill}{rgb}{0.000000,0.000000,0.000000}%
\pgfsetfillcolor{currentfill}%
\pgfsetlinewidth{0.501875pt}%
\definecolor{currentstroke}{rgb}{0.000000,0.000000,0.000000}%
\pgfsetstrokecolor{currentstroke}%
\pgfsetdash{}{0pt}%
\pgfsys@defobject{currentmarker}{\pgfqpoint{-0.048611in}{0.000000in}}{\pgfqpoint{0.000000in}{0.000000in}}{%
\pgfpathmoveto{\pgfqpoint{0.000000in}{0.000000in}}%
\pgfpathlineto{\pgfqpoint{-0.048611in}{0.000000in}}%
\pgfusepath{stroke,fill}%
}%
\begin{pgfscope}%
\pgfsys@transformshift{0.362752in}{0.347569in}%
\pgfsys@useobject{currentmarker}{}%
\end{pgfscope}%
\end{pgfscope}%
\begin{pgfscope}%
\pgftext[x=0.214605in,y=0.318641in,left,base]{\rmfamily\fontsize{6.000000}{7.200000}\selectfont \(\displaystyle 0\)}%
\end{pgfscope}%
\begin{pgfscope}%
\pgfpathrectangle{\pgfqpoint{0.362752in}{0.347569in}}{\pgfqpoint{2.280651in}{1.400426in}} %
\pgfusepath{clip}%
\pgfsetrectcap%
\pgfsetroundjoin%
\pgfsetlinewidth{0.501875pt}%
\definecolor{currentstroke}{rgb}{0.690196,0.690196,0.690196}%
\pgfsetstrokecolor{currentstroke}%
\pgfsetstrokeopacity{0.250000}%
\pgfsetdash{}{0pt}%
\pgfpathmoveto{\pgfqpoint{0.362752in}{0.580973in}}%
\pgfpathlineto{\pgfqpoint{2.643404in}{0.580973in}}%
\pgfusepath{stroke}%
\end{pgfscope}%
\begin{pgfscope}%
\pgfsetbuttcap%
\pgfsetroundjoin%
\definecolor{currentfill}{rgb}{0.000000,0.000000,0.000000}%
\pgfsetfillcolor{currentfill}%
\pgfsetlinewidth{0.501875pt}%
\definecolor{currentstroke}{rgb}{0.000000,0.000000,0.000000}%
\pgfsetstrokecolor{currentstroke}%
\pgfsetdash{}{0pt}%
\pgfsys@defobject{currentmarker}{\pgfqpoint{-0.048611in}{0.000000in}}{\pgfqpoint{0.000000in}{0.000000in}}{%
\pgfpathmoveto{\pgfqpoint{0.000000in}{0.000000in}}%
\pgfpathlineto{\pgfqpoint{-0.048611in}{0.000000in}}%
\pgfusepath{stroke,fill}%
}%
\begin{pgfscope}%
\pgfsys@transformshift{0.362752in}{0.580973in}%
\pgfsys@useobject{currentmarker}{}%
\end{pgfscope}%
\end{pgfscope}%
\begin{pgfscope}%
\pgftext[x=0.163680in,y=0.552045in,left,base]{\rmfamily\fontsize{6.000000}{7.200000}\selectfont \(\displaystyle 10\)}%
\end{pgfscope}%
\begin{pgfscope}%
\pgfpathrectangle{\pgfqpoint{0.362752in}{0.347569in}}{\pgfqpoint{2.280651in}{1.400426in}} %
\pgfusepath{clip}%
\pgfsetrectcap%
\pgfsetroundjoin%
\pgfsetlinewidth{0.501875pt}%
\definecolor{currentstroke}{rgb}{0.690196,0.690196,0.690196}%
\pgfsetstrokecolor{currentstroke}%
\pgfsetstrokeopacity{0.250000}%
\pgfsetdash{}{0pt}%
\pgfpathmoveto{\pgfqpoint{0.362752in}{0.814378in}}%
\pgfpathlineto{\pgfqpoint{2.643404in}{0.814378in}}%
\pgfusepath{stroke}%
\end{pgfscope}%
\begin{pgfscope}%
\pgfsetbuttcap%
\pgfsetroundjoin%
\definecolor{currentfill}{rgb}{0.000000,0.000000,0.000000}%
\pgfsetfillcolor{currentfill}%
\pgfsetlinewidth{0.501875pt}%
\definecolor{currentstroke}{rgb}{0.000000,0.000000,0.000000}%
\pgfsetstrokecolor{currentstroke}%
\pgfsetdash{}{0pt}%
\pgfsys@defobject{currentmarker}{\pgfqpoint{-0.048611in}{0.000000in}}{\pgfqpoint{0.000000in}{0.000000in}}{%
\pgfpathmoveto{\pgfqpoint{0.000000in}{0.000000in}}%
\pgfpathlineto{\pgfqpoint{-0.048611in}{0.000000in}}%
\pgfusepath{stroke,fill}%
}%
\begin{pgfscope}%
\pgfsys@transformshift{0.362752in}{0.814378in}%
\pgfsys@useobject{currentmarker}{}%
\end{pgfscope}%
\end{pgfscope}%
\begin{pgfscope}%
\pgftext[x=0.163680in,y=0.785450in,left,base]{\rmfamily\fontsize{6.000000}{7.200000}\selectfont \(\displaystyle 20\)}%
\end{pgfscope}%
\begin{pgfscope}%
\pgfpathrectangle{\pgfqpoint{0.362752in}{0.347569in}}{\pgfqpoint{2.280651in}{1.400426in}} %
\pgfusepath{clip}%
\pgfsetrectcap%
\pgfsetroundjoin%
\pgfsetlinewidth{0.501875pt}%
\definecolor{currentstroke}{rgb}{0.690196,0.690196,0.690196}%
\pgfsetstrokecolor{currentstroke}%
\pgfsetstrokeopacity{0.250000}%
\pgfsetdash{}{0pt}%
\pgfpathmoveto{\pgfqpoint{0.362752in}{1.047782in}}%
\pgfpathlineto{\pgfqpoint{2.643404in}{1.047782in}}%
\pgfusepath{stroke}%
\end{pgfscope}%
\begin{pgfscope}%
\pgfsetbuttcap%
\pgfsetroundjoin%
\definecolor{currentfill}{rgb}{0.000000,0.000000,0.000000}%
\pgfsetfillcolor{currentfill}%
\pgfsetlinewidth{0.501875pt}%
\definecolor{currentstroke}{rgb}{0.000000,0.000000,0.000000}%
\pgfsetstrokecolor{currentstroke}%
\pgfsetdash{}{0pt}%
\pgfsys@defobject{currentmarker}{\pgfqpoint{-0.048611in}{0.000000in}}{\pgfqpoint{0.000000in}{0.000000in}}{%
\pgfpathmoveto{\pgfqpoint{0.000000in}{0.000000in}}%
\pgfpathlineto{\pgfqpoint{-0.048611in}{0.000000in}}%
\pgfusepath{stroke,fill}%
}%
\begin{pgfscope}%
\pgfsys@transformshift{0.362752in}{1.047782in}%
\pgfsys@useobject{currentmarker}{}%
\end{pgfscope}%
\end{pgfscope}%
\begin{pgfscope}%
\pgftext[x=0.163680in,y=1.018854in,left,base]{\rmfamily\fontsize{6.000000}{7.200000}\selectfont \(\displaystyle 30\)}%
\end{pgfscope}%
\begin{pgfscope}%
\pgfpathrectangle{\pgfqpoint{0.362752in}{0.347569in}}{\pgfqpoint{2.280651in}{1.400426in}} %
\pgfusepath{clip}%
\pgfsetrectcap%
\pgfsetroundjoin%
\pgfsetlinewidth{0.501875pt}%
\definecolor{currentstroke}{rgb}{0.690196,0.690196,0.690196}%
\pgfsetstrokecolor{currentstroke}%
\pgfsetstrokeopacity{0.250000}%
\pgfsetdash{}{0pt}%
\pgfpathmoveto{\pgfqpoint{0.362752in}{1.281187in}}%
\pgfpathlineto{\pgfqpoint{2.643404in}{1.281187in}}%
\pgfusepath{stroke}%
\end{pgfscope}%
\begin{pgfscope}%
\pgfsetbuttcap%
\pgfsetroundjoin%
\definecolor{currentfill}{rgb}{0.000000,0.000000,0.000000}%
\pgfsetfillcolor{currentfill}%
\pgfsetlinewidth{0.501875pt}%
\definecolor{currentstroke}{rgb}{0.000000,0.000000,0.000000}%
\pgfsetstrokecolor{currentstroke}%
\pgfsetdash{}{0pt}%
\pgfsys@defobject{currentmarker}{\pgfqpoint{-0.048611in}{0.000000in}}{\pgfqpoint{0.000000in}{0.000000in}}{%
\pgfpathmoveto{\pgfqpoint{0.000000in}{0.000000in}}%
\pgfpathlineto{\pgfqpoint{-0.048611in}{0.000000in}}%
\pgfusepath{stroke,fill}%
}%
\begin{pgfscope}%
\pgfsys@transformshift{0.362752in}{1.281187in}%
\pgfsys@useobject{currentmarker}{}%
\end{pgfscope}%
\end{pgfscope}%
\begin{pgfscope}%
\pgftext[x=0.163680in,y=1.252258in,left,base]{\rmfamily\fontsize{6.000000}{7.200000}\selectfont \(\displaystyle 40\)}%
\end{pgfscope}%
\begin{pgfscope}%
\pgfpathrectangle{\pgfqpoint{0.362752in}{0.347569in}}{\pgfqpoint{2.280651in}{1.400426in}} %
\pgfusepath{clip}%
\pgfsetrectcap%
\pgfsetroundjoin%
\pgfsetlinewidth{0.501875pt}%
\definecolor{currentstroke}{rgb}{0.690196,0.690196,0.690196}%
\pgfsetstrokecolor{currentstroke}%
\pgfsetstrokeopacity{0.250000}%
\pgfsetdash{}{0pt}%
\pgfpathmoveto{\pgfqpoint{0.362752in}{1.514591in}}%
\pgfpathlineto{\pgfqpoint{2.643404in}{1.514591in}}%
\pgfusepath{stroke}%
\end{pgfscope}%
\begin{pgfscope}%
\pgfsetbuttcap%
\pgfsetroundjoin%
\definecolor{currentfill}{rgb}{0.000000,0.000000,0.000000}%
\pgfsetfillcolor{currentfill}%
\pgfsetlinewidth{0.501875pt}%
\definecolor{currentstroke}{rgb}{0.000000,0.000000,0.000000}%
\pgfsetstrokecolor{currentstroke}%
\pgfsetdash{}{0pt}%
\pgfsys@defobject{currentmarker}{\pgfqpoint{-0.048611in}{0.000000in}}{\pgfqpoint{0.000000in}{0.000000in}}{%
\pgfpathmoveto{\pgfqpoint{0.000000in}{0.000000in}}%
\pgfpathlineto{\pgfqpoint{-0.048611in}{0.000000in}}%
\pgfusepath{stroke,fill}%
}%
\begin{pgfscope}%
\pgfsys@transformshift{0.362752in}{1.514591in}%
\pgfsys@useobject{currentmarker}{}%
\end{pgfscope}%
\end{pgfscope}%
\begin{pgfscope}%
\pgftext[x=0.163680in,y=1.485663in,left,base]{\rmfamily\fontsize{6.000000}{7.200000}\selectfont \(\displaystyle 50\)}%
\end{pgfscope}%
\begin{pgfscope}%
\pgfpathrectangle{\pgfqpoint{0.362752in}{0.347569in}}{\pgfqpoint{2.280651in}{1.400426in}} %
\pgfusepath{clip}%
\pgfsetrectcap%
\pgfsetroundjoin%
\pgfsetlinewidth{0.501875pt}%
\definecolor{currentstroke}{rgb}{0.690196,0.690196,0.690196}%
\pgfsetstrokecolor{currentstroke}%
\pgfsetstrokeopacity{0.250000}%
\pgfsetdash{}{0pt}%
\pgfpathmoveto{\pgfqpoint{0.362752in}{1.747995in}}%
\pgfpathlineto{\pgfqpoint{2.643404in}{1.747995in}}%
\pgfusepath{stroke}%
\end{pgfscope}%
\begin{pgfscope}%
\pgfsetbuttcap%
\pgfsetroundjoin%
\definecolor{currentfill}{rgb}{0.000000,0.000000,0.000000}%
\pgfsetfillcolor{currentfill}%
\pgfsetlinewidth{0.501875pt}%
\definecolor{currentstroke}{rgb}{0.000000,0.000000,0.000000}%
\pgfsetstrokecolor{currentstroke}%
\pgfsetdash{}{0pt}%
\pgfsys@defobject{currentmarker}{\pgfqpoint{-0.048611in}{0.000000in}}{\pgfqpoint{0.000000in}{0.000000in}}{%
\pgfpathmoveto{\pgfqpoint{0.000000in}{0.000000in}}%
\pgfpathlineto{\pgfqpoint{-0.048611in}{0.000000in}}%
\pgfusepath{stroke,fill}%
}%
\begin{pgfscope}%
\pgfsys@transformshift{0.362752in}{1.747995in}%
\pgfsys@useobject{currentmarker}{}%
\end{pgfscope}%
\end{pgfscope}%
\begin{pgfscope}%
\pgftext[x=0.163680in,y=1.719067in,left,base]{\rmfamily\fontsize{6.000000}{7.200000}\selectfont \(\displaystyle 60\)}%
\end{pgfscope}%
\begin{pgfscope}%
\pgftext[x=0.108124in,y=1.047782in,,bottom,rotate=90.000000]{\rmfamily\fontsize{8.000000}{9.600000}\selectfont number of iterations}%
\end{pgfscope}%
\begin{pgfscope}%
\pgfpathrectangle{\pgfqpoint{0.362752in}{0.347569in}}{\pgfqpoint{2.280651in}{1.400426in}} %
\pgfusepath{clip}%
\pgfsetbuttcap%
\pgfsetroundjoin%
\pgfsetlinewidth{0.501875pt}%
\definecolor{currentstroke}{rgb}{0.000000,0.000000,0.000000}%
\pgfsetstrokecolor{currentstroke}%
\pgfsetdash{{5.600000pt}{2.400000pt}}{0.000000pt}%
\pgfpathmoveto{\pgfqpoint{0.466418in}{0.440931in}}%
\pgfpathlineto{\pgfqpoint{0.881082in}{0.464271in}}%
\pgfpathlineto{\pgfqpoint{1.295746in}{0.464271in}}%
\pgfpathlineto{\pgfqpoint{1.710410in}{0.464271in}}%
\pgfpathlineto{\pgfqpoint{2.125074in}{0.440931in}}%
\pgfpathlineto{\pgfqpoint{2.539738in}{0.534293in}}%
\pgfusepath{stroke}%
\end{pgfscope}%
\begin{pgfscope}%
\pgfpathrectangle{\pgfqpoint{0.362752in}{0.347569in}}{\pgfqpoint{2.280651in}{1.400426in}} %
\pgfusepath{clip}%
\pgfsetbuttcap%
\pgfsetroundjoin%
\pgfsetlinewidth{0.501875pt}%
\definecolor{currentstroke}{rgb}{0.000000,0.000000,0.000000}%
\pgfsetstrokecolor{currentstroke}%
\pgfsetdash{{9.600000pt}{2.400000pt}{1.600000pt}{2.400000pt}}{0.000000pt}%
\pgfpathmoveto{\pgfqpoint{0.466418in}{0.464271in}}%
\pgfpathlineto{\pgfqpoint{0.881082in}{0.464271in}}%
\pgfpathlineto{\pgfqpoint{1.295746in}{0.464271in}}%
\pgfpathlineto{\pgfqpoint{1.710410in}{0.464271in}}%
\pgfpathlineto{\pgfqpoint{2.125074in}{0.440931in}}%
\pgfpathlineto{\pgfqpoint{2.539738in}{0.557633in}}%
\pgfusepath{stroke}%
\end{pgfscope}%
\begin{pgfscope}%
\pgfpathrectangle{\pgfqpoint{0.362752in}{0.347569in}}{\pgfqpoint{2.280651in}{1.400426in}} %
\pgfusepath{clip}%
\pgfsetrectcap%
\pgfsetroundjoin%
\pgfsetlinewidth{1.003750pt}%
\definecolor{currentstroke}{rgb}{1.000000,0.000000,0.000000}%
\pgfsetstrokecolor{currentstroke}%
\pgfsetdash{}{0pt}%
\pgfpathmoveto{\pgfqpoint{0.466418in}{0.557633in}}%
\pgfpathlineto{\pgfqpoint{0.881082in}{0.580973in}}%
\pgfpathlineto{\pgfqpoint{1.295746in}{0.604314in}}%
\pgfpathlineto{\pgfqpoint{1.710410in}{0.580973in}}%
\pgfpathlineto{\pgfqpoint{2.125074in}{0.510952in}}%
\pgfpathlineto{\pgfqpoint{2.363298in}{1.757995in}}%
\pgfusepath{stroke}%
\end{pgfscope}%
\begin{pgfscope}%
\pgfpathrectangle{\pgfqpoint{0.362752in}{0.347569in}}{\pgfqpoint{2.280651in}{1.400426in}} %
\pgfusepath{clip}%
\pgfsetbuttcap%
\pgfsetmiterjoin%
\definecolor{currentfill}{rgb}{1.000000,0.000000,0.000000}%
\pgfsetfillcolor{currentfill}%
\pgfsetlinewidth{0.501875pt}%
\definecolor{currentstroke}{rgb}{0.000000,0.000000,0.000000}%
\pgfsetstrokecolor{currentstroke}%
\pgfsetdash{}{0pt}%
\pgfsys@defobject{currentmarker}{\pgfqpoint{-0.027778in}{-0.027778in}}{\pgfqpoint{0.027778in}{0.027778in}}{%
\pgfpathmoveto{\pgfqpoint{-0.027778in}{-0.027778in}}%
\pgfpathlineto{\pgfqpoint{0.027778in}{-0.027778in}}%
\pgfpathlineto{\pgfqpoint{0.027778in}{0.027778in}}%
\pgfpathlineto{\pgfqpoint{-0.027778in}{0.027778in}}%
\pgfpathclose%
\pgfusepath{stroke,fill}%
}%
\begin{pgfscope}%
\pgfsys@transformshift{0.466418in}{0.557633in}%
\pgfsys@useobject{currentmarker}{}%
\end{pgfscope}%
\begin{pgfscope}%
\pgfsys@transformshift{0.881082in}{0.580973in}%
\pgfsys@useobject{currentmarker}{}%
\end{pgfscope}%
\begin{pgfscope}%
\pgfsys@transformshift{1.295746in}{0.604314in}%
\pgfsys@useobject{currentmarker}{}%
\end{pgfscope}%
\begin{pgfscope}%
\pgfsys@transformshift{1.710410in}{0.580973in}%
\pgfsys@useobject{currentmarker}{}%
\end{pgfscope}%
\begin{pgfscope}%
\pgfsys@transformshift{2.125074in}{0.510952in}%
\pgfsys@useobject{currentmarker}{}%
\end{pgfscope}%
\begin{pgfscope}%
\pgfsys@transformshift{2.539738in}{2.681613in}%
\pgfsys@useobject{currentmarker}{}%
\end{pgfscope}%
\end{pgfscope}%
\begin{pgfscope}%
\pgfpathrectangle{\pgfqpoint{0.362752in}{0.347569in}}{\pgfqpoint{2.280651in}{1.400426in}} %
\pgfusepath{clip}%
\pgfsetrectcap%
\pgfsetroundjoin%
\pgfsetlinewidth{1.003750pt}%
\definecolor{currentstroke}{rgb}{0.000000,0.500000,0.000000}%
\pgfsetstrokecolor{currentstroke}%
\pgfsetdash{}{0pt}%
\pgfpathmoveto{\pgfqpoint{0.466418in}{0.464271in}}%
\pgfpathlineto{\pgfqpoint{0.881082in}{0.510952in}}%
\pgfpathlineto{\pgfqpoint{1.295746in}{0.510952in}}%
\pgfpathlineto{\pgfqpoint{1.710410in}{0.510952in}}%
\pgfpathlineto{\pgfqpoint{2.125074in}{0.464271in}}%
\pgfpathlineto{\pgfqpoint{2.539738in}{0.534293in}}%
\pgfusepath{stroke}%
\end{pgfscope}%
\begin{pgfscope}%
\pgfpathrectangle{\pgfqpoint{0.362752in}{0.347569in}}{\pgfqpoint{2.280651in}{1.400426in}} %
\pgfusepath{clip}%
\pgfsetbuttcap%
\pgfsetroundjoin%
\definecolor{currentfill}{rgb}{0.000000,0.500000,0.000000}%
\pgfsetfillcolor{currentfill}%
\pgfsetlinewidth{0.501875pt}%
\definecolor{currentstroke}{rgb}{0.000000,0.000000,0.000000}%
\pgfsetstrokecolor{currentstroke}%
\pgfsetdash{}{0pt}%
\pgfsys@defobject{currentmarker}{\pgfqpoint{-0.027778in}{-0.027778in}}{\pgfqpoint{0.027778in}{0.027778in}}{%
\pgfpathmoveto{\pgfqpoint{0.000000in}{-0.027778in}}%
\pgfpathcurveto{\pgfqpoint{0.007367in}{-0.027778in}}{\pgfqpoint{0.014433in}{-0.024851in}}{\pgfqpoint{0.019642in}{-0.019642in}}%
\pgfpathcurveto{\pgfqpoint{0.024851in}{-0.014433in}}{\pgfqpoint{0.027778in}{-0.007367in}}{\pgfqpoint{0.027778in}{0.000000in}}%
\pgfpathcurveto{\pgfqpoint{0.027778in}{0.007367in}}{\pgfqpoint{0.024851in}{0.014433in}}{\pgfqpoint{0.019642in}{0.019642in}}%
\pgfpathcurveto{\pgfqpoint{0.014433in}{0.024851in}}{\pgfqpoint{0.007367in}{0.027778in}}{\pgfqpoint{0.000000in}{0.027778in}}%
\pgfpathcurveto{\pgfqpoint{-0.007367in}{0.027778in}}{\pgfqpoint{-0.014433in}{0.024851in}}{\pgfqpoint{-0.019642in}{0.019642in}}%
\pgfpathcurveto{\pgfqpoint{-0.024851in}{0.014433in}}{\pgfqpoint{-0.027778in}{0.007367in}}{\pgfqpoint{-0.027778in}{0.000000in}}%
\pgfpathcurveto{\pgfqpoint{-0.027778in}{-0.007367in}}{\pgfqpoint{-0.024851in}{-0.014433in}}{\pgfqpoint{-0.019642in}{-0.019642in}}%
\pgfpathcurveto{\pgfqpoint{-0.014433in}{-0.024851in}}{\pgfqpoint{-0.007367in}{-0.027778in}}{\pgfqpoint{0.000000in}{-0.027778in}}%
\pgfpathclose%
\pgfusepath{stroke,fill}%
}%
\begin{pgfscope}%
\pgfsys@transformshift{0.466418in}{0.464271in}%
\pgfsys@useobject{currentmarker}{}%
\end{pgfscope}%
\begin{pgfscope}%
\pgfsys@transformshift{0.881082in}{0.510952in}%
\pgfsys@useobject{currentmarker}{}%
\end{pgfscope}%
\begin{pgfscope}%
\pgfsys@transformshift{1.295746in}{0.510952in}%
\pgfsys@useobject{currentmarker}{}%
\end{pgfscope}%
\begin{pgfscope}%
\pgfsys@transformshift{1.710410in}{0.510952in}%
\pgfsys@useobject{currentmarker}{}%
\end{pgfscope}%
\begin{pgfscope}%
\pgfsys@transformshift{2.125074in}{0.464271in}%
\pgfsys@useobject{currentmarker}{}%
\end{pgfscope}%
\begin{pgfscope}%
\pgfsys@transformshift{2.539738in}{0.534293in}%
\pgfsys@useobject{currentmarker}{}%
\end{pgfscope}%
\end{pgfscope}%
\begin{pgfscope}%
\pgfpathrectangle{\pgfqpoint{0.362752in}{0.347569in}}{\pgfqpoint{2.280651in}{1.400426in}} %
\pgfusepath{clip}%
\pgfsetrectcap%
\pgfsetroundjoin%
\pgfsetlinewidth{1.003750pt}%
\definecolor{currentstroke}{rgb}{0.000000,0.000000,1.000000}%
\pgfsetstrokecolor{currentstroke}%
\pgfsetdash{}{0pt}%
\pgfpathmoveto{\pgfqpoint{0.466418in}{0.440931in}}%
\pgfpathlineto{\pgfqpoint{0.881082in}{0.440931in}}%
\pgfpathlineto{\pgfqpoint{1.295746in}{0.464271in}}%
\pgfpathlineto{\pgfqpoint{1.710410in}{0.440931in}}%
\pgfpathlineto{\pgfqpoint{2.125074in}{0.417590in}}%
\pgfpathlineto{\pgfqpoint{2.539738in}{0.487612in}}%
\pgfusepath{stroke}%
\end{pgfscope}%
\begin{pgfscope}%
\pgfpathrectangle{\pgfqpoint{0.362752in}{0.347569in}}{\pgfqpoint{2.280651in}{1.400426in}} %
\pgfusepath{clip}%
\pgfsetbuttcap%
\pgfsetmiterjoin%
\definecolor{currentfill}{rgb}{0.000000,0.000000,1.000000}%
\pgfsetfillcolor{currentfill}%
\pgfsetlinewidth{0.501875pt}%
\definecolor{currentstroke}{rgb}{0.000000,0.000000,0.000000}%
\pgfsetstrokecolor{currentstroke}%
\pgfsetdash{}{0pt}%
\pgfsys@defobject{currentmarker}{\pgfqpoint{-0.027778in}{-0.027778in}}{\pgfqpoint{0.027778in}{0.027778in}}{%
\pgfpathmoveto{\pgfqpoint{0.000000in}{0.027778in}}%
\pgfpathlineto{\pgfqpoint{-0.027778in}{-0.027778in}}%
\pgfpathlineto{\pgfqpoint{0.027778in}{-0.027778in}}%
\pgfpathclose%
\pgfusepath{stroke,fill}%
}%
\begin{pgfscope}%
\pgfsys@transformshift{0.466418in}{0.440931in}%
\pgfsys@useobject{currentmarker}{}%
\end{pgfscope}%
\begin{pgfscope}%
\pgfsys@transformshift{0.881082in}{0.440931in}%
\pgfsys@useobject{currentmarker}{}%
\end{pgfscope}%
\begin{pgfscope}%
\pgfsys@transformshift{1.295746in}{0.464271in}%
\pgfsys@useobject{currentmarker}{}%
\end{pgfscope}%
\begin{pgfscope}%
\pgfsys@transformshift{1.710410in}{0.440931in}%
\pgfsys@useobject{currentmarker}{}%
\end{pgfscope}%
\begin{pgfscope}%
\pgfsys@transformshift{2.125074in}{0.417590in}%
\pgfsys@useobject{currentmarker}{}%
\end{pgfscope}%
\begin{pgfscope}%
\pgfsys@transformshift{2.539738in}{0.487612in}%
\pgfsys@useobject{currentmarker}{}%
\end{pgfscope}%
\end{pgfscope}%
\begin{pgfscope}%
\pgfsetrectcap%
\pgfsetmiterjoin%
\pgfsetlinewidth{0.501875pt}%
\definecolor{currentstroke}{rgb}{0.000000,0.000000,0.000000}%
\pgfsetstrokecolor{currentstroke}%
\pgfsetdash{}{0pt}%
\pgfpathmoveto{\pgfqpoint{0.362752in}{0.347569in}}%
\pgfpathlineto{\pgfqpoint{0.362752in}{1.747995in}}%
\pgfusepath{stroke}%
\end{pgfscope}%
\begin{pgfscope}%
\pgfsetrectcap%
\pgfsetmiterjoin%
\pgfsetlinewidth{0.501875pt}%
\definecolor{currentstroke}{rgb}{0.000000,0.000000,0.000000}%
\pgfsetstrokecolor{currentstroke}%
\pgfsetdash{}{0pt}%
\pgfpathmoveto{\pgfqpoint{2.643404in}{0.347569in}}%
\pgfpathlineto{\pgfqpoint{2.643404in}{1.747995in}}%
\pgfusepath{stroke}%
\end{pgfscope}%
\begin{pgfscope}%
\pgfsetrectcap%
\pgfsetmiterjoin%
\pgfsetlinewidth{0.501875pt}%
\definecolor{currentstroke}{rgb}{0.000000,0.000000,0.000000}%
\pgfsetstrokecolor{currentstroke}%
\pgfsetdash{}{0pt}%
\pgfpathmoveto{\pgfqpoint{0.362752in}{0.347569in}}%
\pgfpathlineto{\pgfqpoint{2.643404in}{0.347569in}}%
\pgfusepath{stroke}%
\end{pgfscope}%
\begin{pgfscope}%
\pgfsetrectcap%
\pgfsetmiterjoin%
\pgfsetlinewidth{0.501875pt}%
\definecolor{currentstroke}{rgb}{0.000000,0.000000,0.000000}%
\pgfsetstrokecolor{currentstroke}%
\pgfsetdash{}{0pt}%
\pgfpathmoveto{\pgfqpoint{0.362752in}{1.747995in}}%
\pgfpathlineto{\pgfqpoint{2.643404in}{1.747995in}}%
\pgfusepath{stroke}%
\end{pgfscope}%
\begin{pgfscope}%
\pgfsetbuttcap%
\pgfsetmiterjoin%
\definecolor{currentfill}{rgb}{1.000000,1.000000,1.000000}%
\pgfsetfillcolor{currentfill}%
\pgfsetfillopacity{0.800000}%
\pgfsetlinewidth{1.003750pt}%
\definecolor{currentstroke}{rgb}{0.800000,0.800000,0.800000}%
\pgfsetstrokecolor{currentstroke}%
\pgfsetstrokeopacity{0.800000}%
\pgfsetdash{}{0pt}%
\pgfpathmoveto{\pgfqpoint{0.421086in}{1.100330in}}%
\pgfpathlineto{\pgfqpoint{0.940160in}{1.100330in}}%
\pgfpathquadraticcurveto{\pgfqpoint{0.956826in}{1.100330in}}{\pgfqpoint{0.956826in}{1.116997in}}%
\pgfpathlineto{\pgfqpoint{0.956826in}{1.689662in}}%
\pgfpathquadraticcurveto{\pgfqpoint{0.956826in}{1.706329in}}{\pgfqpoint{0.940160in}{1.706329in}}%
\pgfpathlineto{\pgfqpoint{0.421086in}{1.706329in}}%
\pgfpathquadraticcurveto{\pgfqpoint{0.404419in}{1.706329in}}{\pgfqpoint{0.404419in}{1.689662in}}%
\pgfpathlineto{\pgfqpoint{0.404419in}{1.116997in}}%
\pgfpathquadraticcurveto{\pgfqpoint{0.404419in}{1.100330in}}{\pgfqpoint{0.421086in}{1.100330in}}%
\pgfpathclose%
\pgfusepath{stroke,fill}%
\end{pgfscope}%
\begin{pgfscope}%
\pgfsetbuttcap%
\pgfsetroundjoin%
\pgfsetlinewidth{0.501875pt}%
\definecolor{currentstroke}{rgb}{0.000000,0.000000,0.000000}%
\pgfsetstrokecolor{currentstroke}%
\pgfsetdash{{5.600000pt}{2.400000pt}}{0.000000pt}%
\pgfpathmoveto{\pgfqpoint{0.437752in}{1.643829in}}%
\pgfpathlineto{\pgfqpoint{0.604419in}{1.643829in}}%
\pgfusepath{stroke}%
\end{pgfscope}%
\begin{pgfscope}%
\pgftext[x=0.671086in,y=1.614662in,left,base]{\rmfamily\fontsize{6.000000}{7.200000}\selectfont LU}%
\end{pgfscope}%
\begin{pgfscope}%
\pgfsetbuttcap%
\pgfsetroundjoin%
\pgfsetlinewidth{0.501875pt}%
\definecolor{currentstroke}{rgb}{0.000000,0.000000,0.000000}%
\pgfsetstrokecolor{currentstroke}%
\pgfsetdash{{9.600000pt}{2.400000pt}{1.600000pt}{2.400000pt}}{0.000000pt}%
\pgfpathmoveto{\pgfqpoint{0.437752in}{1.527629in}}%
\pgfpathlineto{\pgfqpoint{0.604419in}{1.527629in}}%
\pgfusepath{stroke}%
\end{pgfscope}%
\begin{pgfscope}%
\pgftext[x=0.671086in,y=1.498462in,left,base]{\rmfamily\fontsize{6.000000}{7.200000}\selectfont IE}%
\end{pgfscope}%
\begin{pgfscope}%
\pgfsetrectcap%
\pgfsetroundjoin%
\pgfsetlinewidth{1.003750pt}%
\definecolor{currentstroke}{rgb}{1.000000,0.000000,0.000000}%
\pgfsetstrokecolor{currentstroke}%
\pgfsetdash{}{0pt}%
\pgfpathmoveto{\pgfqpoint{0.437752in}{1.411429in}}%
\pgfpathlineto{\pgfqpoint{0.604419in}{1.411429in}}%
\pgfusepath{stroke}%
\end{pgfscope}%
\begin{pgfscope}%
\pgfsetbuttcap%
\pgfsetmiterjoin%
\definecolor{currentfill}{rgb}{1.000000,0.000000,0.000000}%
\pgfsetfillcolor{currentfill}%
\pgfsetlinewidth{0.501875pt}%
\definecolor{currentstroke}{rgb}{0.000000,0.000000,0.000000}%
\pgfsetstrokecolor{currentstroke}%
\pgfsetdash{}{0pt}%
\pgfsys@defobject{currentmarker}{\pgfqpoint{-0.027778in}{-0.027778in}}{\pgfqpoint{0.027778in}{0.027778in}}{%
\pgfpathmoveto{\pgfqpoint{-0.027778in}{-0.027778in}}%
\pgfpathlineto{\pgfqpoint{0.027778in}{-0.027778in}}%
\pgfpathlineto{\pgfqpoint{0.027778in}{0.027778in}}%
\pgfpathlineto{\pgfqpoint{-0.027778in}{0.027778in}}%
\pgfpathclose%
\pgfusepath{stroke,fill}%
}%
\begin{pgfscope}%
\pgfsys@transformshift{0.521086in}{1.411429in}%
\pgfsys@useobject{currentmarker}{}%
\end{pgfscope}%
\end{pgfscope}%
\begin{pgfscope}%
\pgftext[x=0.671086in,y=1.382263in,left,base]{\rmfamily\fontsize{6.000000}{7.200000}\selectfont IEpar}%
\end{pgfscope}%
\begin{pgfscope}%
\pgfsetrectcap%
\pgfsetroundjoin%
\pgfsetlinewidth{1.003750pt}%
\definecolor{currentstroke}{rgb}{0.000000,0.500000,0.000000}%
\pgfsetstrokecolor{currentstroke}%
\pgfsetdash{}{0pt}%
\pgfpathmoveto{\pgfqpoint{0.437752in}{1.295229in}}%
\pgfpathlineto{\pgfqpoint{0.604419in}{1.295229in}}%
\pgfusepath{stroke}%
\end{pgfscope}%
\begin{pgfscope}%
\pgfsetbuttcap%
\pgfsetroundjoin%
\definecolor{currentfill}{rgb}{0.000000,0.500000,0.000000}%
\pgfsetfillcolor{currentfill}%
\pgfsetlinewidth{0.501875pt}%
\definecolor{currentstroke}{rgb}{0.000000,0.000000,0.000000}%
\pgfsetstrokecolor{currentstroke}%
\pgfsetdash{}{0pt}%
\pgfsys@defobject{currentmarker}{\pgfqpoint{-0.027778in}{-0.027778in}}{\pgfqpoint{0.027778in}{0.027778in}}{%
\pgfpathmoveto{\pgfqpoint{0.000000in}{-0.027778in}}%
\pgfpathcurveto{\pgfqpoint{0.007367in}{-0.027778in}}{\pgfqpoint{0.014433in}{-0.024851in}}{\pgfqpoint{0.019642in}{-0.019642in}}%
\pgfpathcurveto{\pgfqpoint{0.024851in}{-0.014433in}}{\pgfqpoint{0.027778in}{-0.007367in}}{\pgfqpoint{0.027778in}{0.000000in}}%
\pgfpathcurveto{\pgfqpoint{0.027778in}{0.007367in}}{\pgfqpoint{0.024851in}{0.014433in}}{\pgfqpoint{0.019642in}{0.019642in}}%
\pgfpathcurveto{\pgfqpoint{0.014433in}{0.024851in}}{\pgfqpoint{0.007367in}{0.027778in}}{\pgfqpoint{0.000000in}{0.027778in}}%
\pgfpathcurveto{\pgfqpoint{-0.007367in}{0.027778in}}{\pgfqpoint{-0.014433in}{0.024851in}}{\pgfqpoint{-0.019642in}{0.019642in}}%
\pgfpathcurveto{\pgfqpoint{-0.024851in}{0.014433in}}{\pgfqpoint{-0.027778in}{0.007367in}}{\pgfqpoint{-0.027778in}{0.000000in}}%
\pgfpathcurveto{\pgfqpoint{-0.027778in}{-0.007367in}}{\pgfqpoint{-0.024851in}{-0.014433in}}{\pgfqpoint{-0.019642in}{-0.019642in}}%
\pgfpathcurveto{\pgfqpoint{-0.014433in}{-0.024851in}}{\pgfqpoint{-0.007367in}{-0.027778in}}{\pgfqpoint{0.000000in}{-0.027778in}}%
\pgfpathclose%
\pgfusepath{stroke,fill}%
}%
\begin{pgfscope}%
\pgfsys@transformshift{0.521086in}{1.295229in}%
\pgfsys@useobject{currentmarker}{}%
\end{pgfscope}%
\end{pgfscope}%
\begin{pgfscope}%
\pgftext[x=0.671086in,y=1.266063in,left,base]{\rmfamily\fontsize{6.000000}{7.200000}\selectfont Qpar}%
\end{pgfscope}%
\begin{pgfscope}%
\pgfsetrectcap%
\pgfsetroundjoin%
\pgfsetlinewidth{1.003750pt}%
\definecolor{currentstroke}{rgb}{0.000000,0.000000,1.000000}%
\pgfsetstrokecolor{currentstroke}%
\pgfsetdash{}{0pt}%
\pgfpathmoveto{\pgfqpoint{0.437752in}{1.179030in}}%
\pgfpathlineto{\pgfqpoint{0.604419in}{1.179030in}}%
\pgfusepath{stroke}%
\end{pgfscope}%
\begin{pgfscope}%
\pgfsetbuttcap%
\pgfsetmiterjoin%
\definecolor{currentfill}{rgb}{0.000000,0.000000,1.000000}%
\pgfsetfillcolor{currentfill}%
\pgfsetlinewidth{0.501875pt}%
\definecolor{currentstroke}{rgb}{0.000000,0.000000,0.000000}%
\pgfsetstrokecolor{currentstroke}%
\pgfsetdash{}{0pt}%
\pgfsys@defobject{currentmarker}{\pgfqpoint{-0.027778in}{-0.027778in}}{\pgfqpoint{0.027778in}{0.027778in}}{%
\pgfpathmoveto{\pgfqpoint{0.000000in}{0.027778in}}%
\pgfpathlineto{\pgfqpoint{-0.027778in}{-0.027778in}}%
\pgfpathlineto{\pgfqpoint{0.027778in}{-0.027778in}}%
\pgfpathclose%
\pgfusepath{stroke,fill}%
}%
\begin{pgfscope}%
\pgfsys@transformshift{0.521086in}{1.179030in}%
\pgfsys@useobject{currentmarker}{}%
\end{pgfscope}%
\end{pgfscope}%
\begin{pgfscope}%
\pgftext[x=0.671086in,y=1.149863in,left,base]{\rmfamily\fontsize{6.000000}{7.200000}\selectfont MIN}%
\end{pgfscope}%
\end{pgfpicture}%
\makeatother%
\endgroup%